\documentclass{amsart}
\usepackage{amsfonts, amsbsy, amsmath, amssymb}
\usepackage{tikz}
\usepackage{float}
\usepackage{hyperref}
\usepackage[shortlabels]{enumitem}

\usepackage{mathtools}
\usepackage{breqn}

\usepackage{mmacells}

\newtheorem{thm}{Theorem}[section]
\newtheorem{lem}[thm]{Lemma}

\newtheorem{cor}[thm]{Corollary}
\newtheorem{prop}[thm]{Proposition}
\newtheorem{exmp}[thm]{Example}

\newtheorem{rmk}[thm]{Remark}

\numberwithin{equation}{section}

\theoremstyle{definition}

\allowdisplaybreaks

\newcommand{\f}{\Bbb F}
\newcommand{\A}{\alpha}
\newcommand{\pgl}{\mathrm{PGL}}
\newcommand{\res}{\mathrm{Res}}

\newcommand{\F}{\mathbb{F}}

\newcommand{\aut}{\text{\rm Aut}}
\newcommand{\ch}{\text{\rm char\,}}

\DeclareMathOperator{\trace}{\mathrm{Tr}}

\begin{document}

\title[Classification of Rational Functions of Degree Three]{Classification of Rational Functions of Degree Three over Finite Fields}

\author[X. Hou]{Xiang-dong Hou}
\address{Department of Mathematics and Statistics,
University of South Florida, Tampa, FL 33620}
\email{xhou@usf.edu}

\author[S. Peng]{Siyu Peng}
\address{Department of Mathematics, Massachusetts Institute of Technology, Cambridge, MA 02139}
\email{pengs14@mit.edu}

\author[Y. Qiang]{Yongyu Qiang}
\address{School of Mathematics, Georgia Institute of Technology, Atlanta, GA 30332}
\email{yqiang7@gatech.edu}

\author[S. Zhao]{Shujun Zhao}
\address{Department of Mathematics and Statistics,
University of South Florida, Tampa, FL 33620}
\email{shujunz@usf.edu}



\begin{abstract}
    We study rational functions over finite fields under PGL-equivalence. We say that $f, g \in \F_q(X)$ are \emph{equivalent} if there exist $\psi, \phi \in \f_q(X)$ of degree one such that $g = \psi \circ f \circ \phi$. Most properties of rational functions over finite fields as they appear in theory and applications are preserved under this equivalence. In a recent work, Mattarei and Pizzato classified rational functions of degree three over finite fields in even characteristic.    
In the present paper, we classify all rational functions of degree three over finite fields in odd characteristic. Our approach is based on careful analyses of the value frequencies and the ramification points of the degree three rational functions. The completion of our classification also relies on an explicit formula for the number of equivalence classes of degree three rational functions over finite fields recently obtained by the first author.  
\end{abstract}

\maketitle

\tableofcontents

\section{Introduction}

Let $\f_q$ denote the finite field with $q$ elements. For a nonconstant rational function $f=g/h\in\f_q(X)\setminus\f_q$, where $g,h\in\f_q[X]$ and $\gcd(g,h)=1$, its degree is $\deg f=\max\{\deg g,\deg h\}=[\f_q(X):\f_q(f)]$. We treat the projective linear group $\pgl(2,\f_q)$ as $\{\phi\in\f_q(X):\deg\phi=1\}$ equipped with composition.
Elements of $\pgl(2,\f_q)$ are of the form $(aX+b)/(cX+d)$, where $a,b,c,d\in\f_q$, $ad-bc\ne 0$. Two nonconstant rational functions $f,g\in\f_q(X)\setminus\f_q$ are said to be {\em PGL-equivalent}, or simply, {\em equivalent}, denoted as $f\sim g$, if $g=\psi\circ f\circ \phi$ for some $\psi,\phi\in \pgl(2,\f_q)$. It is easy to see that $f\sim g$ if and only if an $\f_q$-automorphism of $\f_q(X)$ restricts to an isomorphism $\f_q(f)\to\f_q(g)$; if $g=\psi\circ f\circ \phi$, $\psi,\phi\in\pgl(2,\f_q)$, the $\f_q$-automorphism of $\f_q(X)$ is given by $X\mapsto\phi(X)$. Therefore, for the two elements $\psi,\phi\in\pgl(2,\f_q)$ that define the equivalence $f\sim g$, $\phi$ is essential and $\psi$ is passive, that is, given $g$ and $f\circ\phi$, it is easy to determine the unique $\psi\in\pgl(2,\f_q)$, if it exists, such that $g=\psi\circ f\circ\phi$.

The properties of a rational function $f\in\f_q(X)\setminus\f_q$ are very much encoded in the attributes of the function field extension $\f_q(X)/\f_q(f)$, and vice versa. Qualities and quantities pertaining to the function field extension $\f_q(X)/\f_q(f)$ are preserved under the PGL action on $f$. For example, the number of ramification points and their types in this extension are PGL-invariants of the rational function $f$.  

The (two-sided) action by $\pgl(2,\f_q)$ on $\f_q(X)\setminus\f_q$ is quite natural, and the PGL classification of rational functions over $\f_q$ is a fundamental question. However, for a long time, this question has not received much attention. Historically, the study of finite fields put a lot emphasis on polynomials and the interest in rational functions was limited. The situation began to change recently when people realized the important role played by rational functions in the theory and applications of finite fields. It was discovered that one can use rational functions to construct high degree irreducible polynomials over $\f_q$. This construction was explored in a series works by several authors \cite{Ahmadi-FFA-2011, Mattarei-Pizzato-FFA-2017, Mattarei-Pizzato-FFA-2022, Panario-Reis-Wang-JPAA-2020, Reis-JPAA-2020, Reis-FFA-2020, Stichtenoth-Topuzoglu-FFA-2012}. For equivalent rational functions, the constructions are essentially the same. Another active topic in finite fields is permutation polynomials which are polynomials that permute the finite field. Such polynomials have extensive applications in combinatorics and cryptography. Rational functions in $\f_q(X)$ induce functions from the projective line $\Bbb P^1(\f_q)=\f_q\cup\{\infty\}$ to itself. Rational functions that permute the projective line not only generalize permutation polynomials over finite fields but also facilitate new constructions of permutation polynomials. Permutation rational functions of $\Bbb P^1(\f_q)$ come in equivalence classes, and such equivalence classes have been determined for permutation rational functions of degree $\le 4$.
There is a rapidly growing volume of literature on permutation polynomials and permutation rational functions over finite fields; we list a few works \cite{Ding-Zieve-AA-2022, Ferraguti-Micheli-DCC-2020, Hou-FFA-2015a, Hou-CA-2021, Sze-Thesis-2023, Wang-2019}, and more references can be found therein. 

All degree one rational functions over $\f_q$ are equivalent. The classification of degree two rational functions over $\f_q$ is also trivial \cite{Hou-AA-2025}: there are precisely two equivalence classes which are represented by $X^2$ and $X^2+X$ for even $q$ and by $X^2$ and $(X^2+b)/X$, where $b$ is a nonsquare of $\f_q$, for odd $q$. The case of degree three is difficult. Very recently, Mattarei and Pizzato \cite{Mattarei-Pizzato-JAA-2024} classified degree three rational functions over $\f_q$ for even $q$ with an approach based on analysis of ramification points. For odd $q$, they gave an upper bound for the number of equivalence classes of degree three rational function. Let $\frak N(q,n)$ denote the number of equivalence classes of rational functions of degree $n$ over $\f_q$. The results of \cite{Mattarei-Pizzato-JAA-2024} give that  
\[
\frak N(q,3)
\begin{cases}
=2(q+1)&\text{if $q$ is even and a square},\cr
=2q&\text{if $q$ is even and a nonsquare},\cr
\le 4q&\text{if $q$ is odd}.
\end{cases}
\]
An explicit formula of $\frak N(q,n)$ for all $q$ and $n$ has been found in \cite{Hou-AA-2025}. For $n=3$, we have
\begin{equation}\label{N3}
\frak N(q,3)=
\begin{cases}
2(q+1)&\text{if}\ q\equiv 1,4\pmod 6,\cr
2q&\text{if}\ q\equiv 2,5\pmod 6,\cr
2q+1&\text{if}\ q\equiv 3\pmod 6.
\end{cases}
\end{equation}

The main objective of the present paper is to provide a complete classification of degree three rational functions over $\f_q$ with odd $q$. Our strategy is to find a collection of $\frak N(q,3)$ degree 3 rational functions  which are pairwise nonequivalent. Here is a brief description of our approach. We divide degree 3 rational functions $f\in\f_q(X)$ into three classes, viewing them as functions from $\Bbb P^1(\f_q)$ to itself.
\begin{itemize}
\item[(I)] $|f^{-1}(\alpha)|=1$ for all $\alpha\in\Bbb P^1(\f_q)$.

\medskip
\item[(II)] $|f^{-1}(\alpha)|=2$ for some $\alpha\in\Bbb P^1(\f_q)$.

\medskip
\item[(III)] The remaining case: $|f^{-1}(\alpha)|=3$ for some $\alpha\in\Bbb P^{-1}(\f_q)$, but $|f^{-1}(\beta)|\ne 2$ for all $\beta\in\Bbb P^1(\f_q)$.
\end{itemize}
Clearly, functions in different classes are nonequivalent.

Classification for Class I is known \cite{Ferraguti-Micheli-DCC-2020}. For functions in Class II, we are able to bring them to canonical forms and show that the canonical forms are pairwise nonequivalent. Let $\frak N_{\rm I}(q,3)$, $\frak N_{\rm II}(q,3)$ and $\frak N_{\rm III}(q,3)$ denote the number of equivalence classes in Classes I, II, and III, respectively. We know that 
\begin{equation}\label{n1}
\frak N_{\rm I}(q,3)=
\begin{cases}
1&\text{if}\ q\not\equiv 3\pmod 6,\cr
2&\text{if}\ q\equiv 3\pmod 6,
\end{cases}
\end{equation}
and 
\begin{equation}\label{n2}
\frak N_{\rm II}(q,3)=
\begin{cases}
q+2&\text{if}\ q\equiv 1\pmod 3,\cr
q&\text{if}\ q\not\equiv 1\pmod 3.
\end{cases}
\end{equation}
Combining \eqref{N3} with \eqref{n1} and \eqref{n2} gives
\begin{equation}\label{n3}
\frak N_{\rm III}(q,3)=q-1.
\end{equation}
To classify functions in Class III, we first assume that $p:=\ch\f_q\ge 5$. We find a set of functions in Class III and show that the functions in the set are pairwise nonequivalent and that the set has cardinality $\ge q-1$. This turns out to be quite difficult. To establish pairwise nonequivalence, we construct some variants based the types of the ramification points of the functions. The difficulty that we encounter here is the fact that degree 3 rational functions over $\f_q$ in odd characteristic can have up to four ramification points while in even characteristic, they can have at most two ramification points. To prove that the set of functions has cardinality $\ge q-1$, we use the Hasse-Weil bound on the number of rational points on absolutely irreducible algebraic curves over finite fields. For $p=3$, the same approach with a few modifications also produces the classification for Class III.

The paper is organized as follows: In Section~2, we recall the known classification for Class I. In Section~3, we state and prove the classification for Class II. In preparation for the classification for Class III, two unrelated topics, ramification points of rational functions over finite fields and the cross ratio, are discussed in Sections~4 and 5 separately. Section~6 is devoted to the classification for Class III with $p\ge 5$. Here, Class III is divided into several subclasses according to the types of the ramification points of the rational functions. Class III with $p=3$ is resolved in Section~7. We construct several few invariants in Class III which allow us to separate at least $q-1$ equivalence classes. Since $\frak M_{\text{III}}(q,3)=q-1$, we conclude, a priori, that these invariants completely determine the equivalence classes in Class III. One wonders if it is possible to prove this conclusion directly, i.e., without using the knowledge of $\frak N_{\text{III}}(q,3)$. The answer is ``almost'', and explore this alternative in Section~8. As concrete examples, the classification of degree 3 rational functions over $\f_{3^3}$ and over $\f_{5^2}$ are exhibited in Section~9. Our approach also works for the classification of degree 3 rational functions over finite fields in even characteristic. Appendix ~A provides a comparison between the result in \cite{Mattarei-Pizzato-JAA-2024} and ours. Appendix~B contains some simple Mathematica codes for readers to verify the computational results in the paper.

The notion of equivalence extends to nonconstant rational functions over an arbitrary field $\f$. An element of $\pgl(2,\f)$ is represented either by an invertible matrix $\left[\begin{smallmatrix}a&b\cr c&d\end{smallmatrix}\right]$ or by a degree one rational function $(aX+bY)/(cX+d)$. For $f(X)=P(X)/Q(X)\in\f(X)\setminus\f$, where $P,Q\in\f[X]$, $\gcd(P,Q)=1$, define
\[
\mathcal S(f)=\langle P,Q\rangle_{\f},
\]
where $\langle \ \rangle_{\f}$ denotes the $\f$-span. For $f,g\in\f(X)\setminus\f$, it is clear that $f\sim g$ if and only if $\mathcal S(g)=\mathcal S(f\circ \phi)$ for some $\phi\in\pgl(2,\f)$. In fact, $\mathcal S(g)=\mathcal S(f\circ \phi)$ means that $g=\psi\circ f\circ \phi$ for some $\psi\in\pgl(2,\f)$. Note that when $\phi$ is given, $\psi$ is unique since $f\circ\phi$ is transcendental over $\f$.


\section{Class I}

Degree 3 rational functions over $\f_q$ that permute the projective line $\Bbb P^1(\f_q)$ were classified by Ferraguti and Micheli in \cite{Ferraguti-Micheli-DCC-2020} using the Chebotarev theorem for function fields. The result is shown in Table~\ref{tb1}, where  the equivalence classes do not depend on the choices of $a,b,a_1,a_2,b_0$. For a different approach to the question, see \cite{Hou-CC-2021}.

\begin{table}[h]
\caption{Classification of degree 3 rational functions over $\f_q$ that permute $\Bbb P^1(\f_q)$}\label{tb1}
   \renewcommand*{\arraystretch}{1.4}
    \centering
     \begin{tabular}{ccl}
         \hline
         $q$ & no. equiv. cls. & \hfil representatives\\ \hline
         $q\equiv 1\pmod 6$ & 1 & $\displaystyle\frac{X^3+aX}{bX^2+1},\,ab=9, -b$ is
         a nonsquare\\
         $q\equiv 2,5\pmod 6$ & 1 & $X^3$\\
         $q\equiv 3\pmod 6$ & 2 & $X^3, X^3+aX,\, -a$ is a nonsquare\\
         $q\equiv 4\pmod 6$ & 1 & $\displaystyle\frac{X^3+a_2X^2+a_1X}{X^2+X+b_0}$,\\
         && $\text{Tr}_{q/2}(b_0)=1, a_1=b_0+b_0^{-1}, a_2=1+b_0^{-1}$\\
         \hline
     \end{tabular}
\end{table}


\section{Class II}

Let $f(X)$ lie in Class II. Taking a suitable composition with functions in $\pgl(2,\f_q)$, we may assume that $f^{-1}(0) = \{0, 1\}$.
Then $f(X)=P(X)/Q(X)$, where $P,Q\in \f_q[X]$, $\gcd(P,Q)=1$ and $P(0)=P(1)=0$. If $P=X(X-1)$, then $\deg Q=3$ and hence $f(\infty)=0$, which is a contradiction. So $P=X^2(X-1)$ or $X(X-1)^2$; we may assume that $P=X^2(X-1)$. Then
\[
f(X)\sim f(X^{-1})=\frac{X^{-2}(X^{-1}-1)}{Q(X^{-1})}=\frac{1-X}{X^3Q(X^{-1})}\sim\frac{X^3Q(X^{-1})}{1-X}\sim\frac{X^3+sX^2+t}X,
\]
where $s\in\f_q$ and $t\in\f_q^*$.

If $s =0$, then for any $a \in \f_q^*$, 
\[
f\sim \frac{X^3+t}X\sim \frac{(aX)^3+t}X\sim\frac{X^3+a^{-3}t}X.
\]
Let $\mathcal C$ be a system of representatives of the cosets of $(\f_q^*)^{(3)}=\{x^3:x\in\f_q^*\}$ in $\f_q^*$. Then in the above, we may assume $t\in\mathcal C$.  

If $s\ne 0$, 
\[
f\sim f(sX)\sim\frac{(sX)^3+s(sX)^2+t}X\sim\frac{X^3+X^2+t/s^3}X.
\]
We may then replace $t/s^3$ with $t$. Thus, we know that all functions in Class II are equivalent to one of the following forms: 
\[
\frac{X^3 + t}{X},\  t\in\mathcal C, \quad  \frac{X^3 + X^2 + t}{X},\ t\in\f_q^*.
\]
In fact, these canonical forms are pairwise inequivalent and thus classifies functions in Class II. 

\begin{thm}\label{Theorem 1}
In Class II, the equivalence classes are represented by
\[
\frac{X^3+t}X,\ t\in\mathcal C\quad \text{and}\quad \frac{X^3+X^2+t}X,\ t\in\f_q^*.
\]
The number of equivalence classes in Class II is
\[
\begin{cases}
q+2&\text{if}\ q\equiv 1\pmod 3,\cr
q&\text{if}\ q\not\equiv 1\pmod 3.
\end{cases}
\]
\end{thm}

\medskip
The proof of Theorem~\ref{Theorem 1} is tedious. We divide the theorem into three propositions and prove them separately.

\begin{prop}\label{Proposition 1.1}
Let $t,t'\in\mathcal C$ be such that
\[
\frac{X^3+t}X\sim\frac{X^3+t'}X.
\]
Then $t=t'$.
\end{prop}

\begin{proof}
Let $\phi=(aX+b)/(cX+d)\in\pgl(2,\f_q)$ be such that
\begin{equation}\label{P3.2eqS}
\mathcal S\Bigl(\frac{X^3+t}X\circ\phi\Bigr)=\mathcal S\Bigl(\frac{X^3+t'}X\Bigr)=\langle X^3+t',\, X\rangle_{\f_q}.
\end{equation}
Write
\[
\frac{X^3+t}X\circ\phi=\frac{A_3X^3+A_2X^2+A_1X+A_0}{B_3X^3+B_2X^2+B_1X+B_0},
\]
where
\begin{equation}\label{P3.2A}
\begin{cases}
A_0=b^3 + d^3 t,\cr
A_1=3 (a b^2+c d^2 t),\cr
A_2=3 (a^2 b + c^2 d t)\cr
A_3=a^3+c^3 t.
\end{cases}
\end{equation}
\begin{equation}\label{P3.2B}
\begin{cases}
B_0=b d^2,\cr
B_1=d (2 b c + a d),\cr
B_2=c (b c+2 a d),\cr
B_3=a c^2,
\end{cases}
\end{equation}
and $(A_3,B_3)\ne(0,0)$. We have
\[
B_3(A_3X^3+\cdots+A_0)-A_3(B_3X^3+\cdots+B_0)=(ad-bc)(f_0+f_1X+f_2X^2),
\]
where
\begin{equation}\label{P3.2f0-2}
\begin{cases}
f_0=-a b^2 c-a^2 b d+c^2 d^2 t,\cr
f_1=-3 a^2 b c - a^3 d + 2 c^3 d t,\cr
f_2=c (-2 a^3 + c^3 s).
\end{cases}
\end{equation}
It follows from \eqref{P3.2eqS} that
\begin{equation}\label{P1.1-1}
\begin{cases}
A_2=0,\cr
B_2=0,\cr
f_0=0,\cr
f_2=0.
\end{cases}
\end{equation}
Moreover,
\begin{equation}\label{P1.1-2}
t'=\begin{cases}
\displaystyle\frac{B_0}{B_3}=\displaystyle\frac{bd^2}{ac^2}&\text{if}\ ac\ne0, \vspace{0.6em}\cr
\displaystyle\frac{A_0}{A_3}=\displaystyle\frac{b^3+d^3t}{a^3+c^3t}&\text{if}\ ac=0.
\end{cases}
\end{equation}
Note that $B_2=0$ gives
\begin{equation}\label{B2=0}
c(bc+2ad)=0.
\end{equation}

\medskip
{\bf Case 1.} Assume $c=0$. Then $ad\ne 0$ and $f_0=-a^2bd$, whence $b=0$. It follows that 
\[
t'=\Bigl(\frac da\Bigr)^3t,
\]
whence $t'=t$.

\medskip
{\bf Case 2.} 
Assume $c\ne 0$. By \eqref{B2=0}, we have $bc+2ad=0$. 
If $p=2$, then $f_2=c^4t\ne 0$, which is a contradiction.
Assume $p\ge 3$. Then $b\ne 0$ (otherwise, $ad=0$ and hence $ad-bc=0$). Thus $c=-2ad/b$. It follows that 
\[
f_0=\frac{a^2 d \left(b^3+4 d^3 t\right)}{b^2}.
\]
Therefore, $t=-b^3/4d^3$. We also have
\[
t'=\frac{bd^2}{ac^2}=\frac{bd^2}a\cdot\Bigl(\frac {-b}{2ad}\Bigr)^2=\frac{b^3}{4a^3}.
\]
Therefore, $t'=(-d/a)^3t$, whence $t'=t$.
\end{proof}

\begin{prop}\label{Proposition 1.2}
Let $t,t'\in\f_q^*$ be such that 
\[
\frac{X^3+X^2+t}X\sim\frac{X^3+X^2+t'}X.
\]
Then $t=t'$.
\end{prop}

\begin{proof}
The computations are similar to those in the proof of Proposition~\ref{Proposition 1.1}. Let $\phi=(aX+b)/(cX+d)\in\pgl(2,\f_q)$ be such that
\begin{equation}\label{P1.2-1}
\mathcal S\Bigl(\frac{X^3+X^2+t}X\circ\phi\Bigr)=\mathcal S\Bigl(\frac{X^3+X^2+t'}X\Bigr)=\langle X^3+X^2+t',\,X\rangle_{\f_q}.
\end{equation}
Write
\[
\frac{X^3+X^2+t}X\circ\phi=\frac{A_3X^3+A_2X^2+A_1X+A_0}{B_3X^3+B_2X^2+B_1X+B_0},
\]
where
\begin{equation}\label{A}
\begin{cases}
A_0=b^3+b^2 d+d^3 t,\cr
A_1=3 a b^2+2 a b d+b^2 c+3 c d^2 t,\cr
A_2=3 a^2 b+a^2 d+2 a b c+3c^2 d t,\cr
A_3=a^3+a^2 c+c^3 t.
\end{cases}
\end{equation}
\begin{equation}\label{B}
\begin{cases}
B_0=b d^2,\cr
B_1=d (2 b c+ad),\cr
B_2=c (bc+2 a d),\cr
B_3=a c^2
\end{cases}
\end{equation}
and $(A_3,B_3)\ne (0,0)$. 
We have 
\[
B_3(A_3X^3+\cdots+A_0)-A_3(B_3X^3+\cdots+B_0)=(ad-bc)(f_0+f_1X+f_2X^2),
\]
where
\begin{equation}\label{f0-2}
\begin{cases}
f_0=-a b^2 c-a^2 b d-a b c d+c^2 d^2 t,\cr
f_1=-3 a^2 b c-a b c^2-a^3 d-a^2 c d+2 c^3 d t,\cr
f_2=c (-2 a^3-a^2 c+c^3 t).
\end{cases}
\end{equation}
It follows from that \eqref{P1.2-1} that
\[
\begin{cases}
A_3=A_2,\cr
B_3=B_2,\cr
f_0=0,\cr
f_2=0.
\end{cases}
\]
Moreover,
\begin{equation}\label{t'}
t'=\begin{cases}
\displaystyle\frac{B_0}{A_3}=\frac{b^3+b^2d+d^3t}{a^3+a^2c+c^3t}&\text{if}\ a^3+a^2c+c^3t\ne 0,\vspace{0.4em}\cr
\displaystyle\frac{B_0}{B_3}=\frac{bd^2}{ac^2}&\text{if}\ ac\ne 0.
\end{cases}
\end{equation}
Let $\res(f_0,f_2;t)$ denote the resultant of $f_0$ and $f_2$ with respect to $t$. We have 
\[
0=\res(f_0,f_2;t)=ac^3(bc-ad)(bc+2ad+cd),
\]
whence
\begin{equation}\label{res}
ac(bc+2ad+cd)=0.
\end{equation}

\medskip
{\bf Case 1.} Assume that $a=0$. Then $c\ne 0$ and $f_2=c^4t\ne 0$, which is a contradiction.

\medskip
{\bf Case 2.} Assume that $a\ne0$ but $c=0$. Then $d\ne 0$ and $f_0=-a^2bd$, whence $b=0$. Now $A_3=A_2$ gives $a=d$. Then by \eqref{t'}, $t'=t$.

\medskip
{\bf Case 3.} Assume that $ac\ne0$. Then by \eqref{res}, $bc+2ad+cd=0$. On the other hand, $B_3=B_2$ gives $bc+2ad=ac$. It follows that $a=-d$.

\medskip
{\bf Case 3.1.} Assume $p>2$.
If $b+d=0$, then $bc+2ad+cd=2ad=-2a^2\ne 0$, which is a contradiction.
Assume that $b+d\ne0$. Then it follows from $bc+2ad+cd=0$ that
\begin{equation}\label{c}
c=-\frac{2ad}{b+d}.
\end{equation}
Using the substitutions $a=-d$ and \eqref{c} in $f_2=0$ gives
\[
t=\frac{2a^3+a^2c}{c^3}=-\frac{b(b+d)^3}{4d^3}
\]
On the other hand, by \eqref{t'} and the same substitutions,
\[
t'=\frac{bd^2}{ac^2}=-\frac{b(b+d)^3}{4d^3}.
\]
Hence $t'=t$.

\medskip
{\bf Case 3.2.} Assume $p=2$. Then \eqref{res} gives $d=b$, and $f_2=0$ gives 
\[
t=\frac{a^2}{c^2}.
\]
Since we also have $a=d=b$, by \eqref{t'},
\[
t'=\frac{bd^2}{ac^2}=\frac{a^2}{c^2}.
\]
\end{proof}

\begin{prop}\label{Proposition 1.3}
For $t,t'\in\f_q^*$,
\[
\frac{X^3+X^2+t}X\not\sim\frac{X^3+t'}X.
\]
\end{prop}

\begin{proof}
Assume to the contrary that there exists $\phi=(aX+b)/(cX+d)\in\pgl$ be such that
\[
\mathcal S\Bigl(\frac{X^3+X^2+t}X\circ\phi\Bigr)=\mathcal S\Bigl(\frac{X^3+t'}X\Bigr)=\langle X^3+t',\,X\rangle_{\f_q}.
\]
Write
\[
\frac{X^3+X^2+t}X\circ\phi=\frac{A_3X^3+A_2X^2+A_1X+A_0}{B_3X^3+B_2X^2+B_1X+B_0},\quad A_3B_3\ne 0,
\]
where $A_0,\dots,A_3,B_0,\dots,B_3$ are given by \eqref{A} and \eqref{B}. We have
\[
B_3(A_3X^3+\cdots+A_0)-A_3(B_3X^3+\cdots+B_0)=(ad-bc)(f_0+f_1X+f_2X^2),
\]
where $f_0,f_1,f_2$ are given by \eqref{f0-2}. We have
\[
\begin{cases}
A_2=0,\cr
B_2=0,\cr
f_0=0,\cr
f_2=0.
\end{cases}
\]
Moreover, \eqref{res} still holds. 
If $a=0$, then $c\ne 0$ and $f_2=c^4t\ne 0$, which is a contradiction.
If $a\ne 0$ but $c=0$, then $d\ne 0$ and $f_0=-a^2bd$, whence $b=0$. Then $A_2=a^2d\ne 0$, which is a contradiction.
Now assume $ac\ne 0$. Then \eqref{res} and $B_2=0$ together imply that $cd=0$, i.e., $d=0$, and \eqref{res} further implies that $bc=0$, i.e., $b=0$. We have $b=d=0$, which is a contradiction.
\end{proof}


\section{Ramification Points}\label{sec4}

For a function field $F$ (a field of algebraic functions in one variable over a constant field $K$), let $\Bbb P_F$ denote the set of places of $F$.

Let $f\in\f_q(X)\setminus\f_q$, and consider the extension $\f_q(X)/\f_q(f)$ as function fields over $\f_q$. 
\[
\begin{tikzpicture}[xscale=1, yscale=0.8]

\node at (0,0) {$\f_q(f)$};
\node at (0,2) {$\f_q(X)$};

\draw  (0,0.5) -- (0,1.5);
\end{tikzpicture}
\]
Degree one places (rational places) of $\f_q(X)$ correspond to the elements $\f_q\cup\{\infty\}$: $a\in\f_q$ corresponds to the zero of $X-a$, denoted by $\frak P_a$, and $\infty$ corresponds to the pole of $X$, denoted by $\frak P_\infty$. The same is true for the rational function field $\f_q(f)/\f_q$: $b\in\f_q$ corresponds to the zero of $f-b$, denoted by $\frak p_b$, and $\infty$ corresponds to the pole of $f$, denoted by $\frak p_\infty$. For $a\in\f_q\cup\{\infty\}$, $\frak P_a\mid\frak p_{f(a)}$. The ramification index $e(\frak P_a\,|\,\frak p_{f(a)})$ can be computed as follows:
\[
e(\frak P_a\,|\,\frak p_{f(a)})=
\begin{cases}
\nu_{X-a}(f(X)-f(a))&\text{if}\ a\ne\infty,\ f(a)\ne\infty,\vspace{0.4em}\cr
\displaystyle\nu_{X-a}\Bigl(\frac1{f(X)}\Bigr)&\text{if}\ a\ne\infty,\ f(a)=\infty,\vspace{0.4em}\cr
\nu_X(f(1/X)-f(\infty))&\text{if}\ a=\infty,\ f(a)\ne\infty,\vspace{0.4em}\cr
\displaystyle\nu_X\Bigl(\frac1{f(1/X)}\Bigr)&\text{if}\ a=\infty,\ f(a)=\infty.
\end{cases}
\]
In the above, $\nu_{X-a}$ is the $(X-a)$-adic valuation of rational functions, i.e., the largest integer $m$ such that $(X-a)^m$ divides the numerator of the rational function. If $e(\frak P_a\,|\,\frak p_{f(a)})=e>1$, we call $a$ a ramification point of $f$ of index $e$ and degree 1, denoted as type $e/1$, and we call $f(a)$ the corresponding branch point of $a$. When $a,f(a)\ne\infty$, it is easy to see that $e(\frak P_a\,|\,\frak p_{f(a)})>1$ (i.e., $\frak P_a$ is ramified in $\f_q(X)/\f_q(f)$) if and only if $f'(a)=0$.

Now let $\frak P\in\Bbb P_{\f_q(X)}$ be a place of degree $n>1$, hence $\frak P$ is generated by an irreducible polynomial $p(X)\in\f_q[X]$ of degree $n$. Let $r_1,\dots,r_n\in\overline\f_q$ be the roots of $p(X)$, and consider the constant field extensions $\overline\f_q(X)/\f_q(X)$ and $\overline\f_q(f)/\f_q(f)$.
\[
\begin{tikzpicture}[xscale=0.8, yscale=0.8]

\node at (0,0) {$\f_q(f)$};
\node at (0,2) {$\f_q(X)$};
\node at (2,2) {$\overline\f_q(f)$};
\node at (2,4) {$\overline\f_q(X)$};

\draw  (0,0.5) -- (0,1.5);
\draw  (2,2.5) -- (2,3.5);
\draw  (0.5,0.5) -- (1.5,1.5);
\draw  (0.5,2.5) -- (1.5,3.5);
\end{tikzpicture}
\]
The conorm of $\frak P$ with respect to $\overline\f_q(X)/\f_q(X)$ is given by $\text{Con}_{\overline\f_q(X)/\f_q(X)}(\frak P)=\frak Q_{r_1}+\cdots+\frak Q_{r_n}$, where $\frak Q_{r_i}\in\Bbb P_{\overline\f_q(X)}$ is the zero of $X-r_i$. Since the constant field extensions $\overline\f_q(X)/\f_q(X)$ and $\overline\f_q(f)/\f_q(f)$ are unramified (\cite[Theorem~3.6.3]{Stichtenoth-2009}), the ramification index of $\frak P$ in $\f_q(X)/\f_q(f)$ equals that of $\frak Q_{r_i}$ in $\overline\f_q(X)/\overline\f_q(f)$, that is,
\[
e(\frak P\,|\, \frak p)=e(\frak Q_{r_i}\,|\, \frak q_{f(r_i)}),
\]
where $\frak p\in\Bbb P_{\f_q(f)}$ is the place below $\frak P$ and $\frak q_{f(r_i)}\in\Bbb P_{\overline\f_q(f)}$ denotes the zero of $f-f(r_i)$ if $f(r_i)\ne\infty$ or the pole of $f$ if $f(r_i)=\infty$. 
\[
\begin{tikzpicture}[xscale=0.8, yscale=0.8]

\node at (0,0) {$\frak p$};
\node at (0,2) {$\frak P$};
\node at (2,2) {$\frak q_{f(r_i)}$};
\node at (2,4) {$\frak Q_{r_i}$};

\draw  (0,0.5) -- (0,1.5);
\draw  (2,2.5) -- (2,3.5);
\draw  (0.5,0.5) -- (1.5,1.5);
\draw  (0.5,2.5) -- (1.5,3.5);
\end{tikzpicture}
\]
If $e(\frak P\,|\, \frak p)=e>1$, we call each $r_i$ a ramification point of $f$ of index $e$ and degree $n$, denoted as type $e/n$, and we call $f(r_i)$ the corresponding branch point of $r_i$.

\medskip
\noindent Note: In the above, it suffices to consider the extensions $\f_{q^n}(X)/\f_q(X)$ and $\f_{q^n}(f)/\f_q(f)$. 

\medskip
Let $f_1,f_2\in\f_q(X)\setminus\f_q$ be PGL-equivalent over $\f_q$, say $f_2=\phi\circ f_1\circ \psi$, where $\phi,\psi\in\pgl(2,\f_q)$. Let $\sigma_\psi\in\aut(\f_q(X)/\f_q)$ be defined by $\sigma_\psi(X)=\psi(X)$, and let $\tau=\sigma_\psi|_{\f_q(f_1)}$. Then for each $\frak P\in\Bbb P_{\f_q(X)}$ and $\frak p\in\Bbb P_{\f_q(f_1)}$ with $\frak P\mid \frak p$, we have $\sigma_\psi(\frak P)\in\Bbb P_{\f_q(X)}$, $\tau(\frak p)\in\Bbb P_{\f_q(f_2)}$, and $\sigma_\psi(\frak P)\mid \tau(\frak p)$. Moreover, $\frak P\,|\,\frak p$ and $\sigma_\psi(\frak P)\,|\,\tau(\frak p)$ have the same ramification index and relative degree. 
\begin{equation}\label{diagm}
\begin{tikzpicture}[xscale=1, yscale=0.8, baseline=2em]

\node at (0,0) {$\f_q(f_1)$};
\node at (0,2) {$\f_q(X)$};
\node at (2,0) {$\f_q(f_2)$};
\node at (2,2) {$\f_q(X)$};

\draw (0,0.5) -- (0,1.5);
\draw (2,0.5) -- (2,1.5);
\draw [->] (0.7,0) to (1.3,0);
\draw [->] (0.7,2) to (1.3,2);

\node [above] at (1,0) {$\scriptstyle\tau$};
\node [above] at (1,2) {$\scriptstyle\sigma_\psi$};
\end{tikzpicture}
\end{equation}
The rational function fields in the diagram~\eqref{diagm} have constant field extensions $\overline\f_q(X)/\f_q(X)$ and $\overline\f_q(f_i)/\f_q(f_i)$. The isomorphisms $\sigma_\psi$ and $\tau$ extend to these extension fields such that the following diagram commutes, where the unlabeled maps are inclusions.
\[
\begin{tikzpicture}[xscale=1.5, yscale=1.2]

\node at (0,0) {$\f_q(f_1)$};
\node at (0,2) {$\f_q(X)$};
\node at (2,0) {$\f_q(f_2)$};
\node at (2,2) {$\f_q(X)$};

\draw [->] (0,0.5) -- (0,1.5);
\draw [->] (2,0.5) -- (2,1.5);
\draw [->] (0.5,0) to (1.5,0);
\draw [->] (0.5,2) to (1.5,2);

\node [above] at (1,0) {$\scriptstyle\tau$};
\node [above] at (1.3,2) {$\scriptstyle\sigma_\psi$};

\node at (1,1) {$\overline\f_q(f_1)$};
\node at (1,3) {$\overline\f_q(X)$};
\node at (3,1) {$\overline\f_q(f_2)$};
\node at (3,3) {$\overline\f_q(X)$};

\draw [->] (1,1.5) -- (1,2.5);
\draw [->] (3,1.5) -- (3,2.5);
\draw [->] (1.5,1) to (2.5,1);
\draw [->] (1.5,3) to (2.5,3);

\node [above] at (2.3,1) {$\scriptstyle\tau$};
\node [above] at (2,3) {$\scriptstyle\sigma_\psi$};

\draw [->] (0.3,0.3) -- (0.7,0.7);
\draw [->] (2.3,0.3) -- (2.7,0.7);
\draw [->] (0.3,2.3) -- (0.7,2.7);
\draw [->] (2.3,2.3) -- (2.7,2.7);

\end{tikzpicture}
\]
Moreover, $\psi:\overline\f_q\cup\{\infty\}\to\overline\f_q\cup\{\infty\}$ maps the ramification points of $f_2$ to those of $f_1$ and preserves the types (index/degree) of the ramification points.

Given $f\in\f_q(X)\setminus\f_q$ with ramification points $r_1,\dots,r_k\in\overline\f_q\cup\{\infty\}$, let $e_i$ denote the ramification index of $r_i$ and $d_i$ denote the the degree of $r_i$ over $\f_q$. (The degree of $\infty$ over $\f_q$ is defined to be 1.) We call the unordered sequence $(e_1/d_1,\dots,e_k/d_k)$ the {\em ramification type} of $f$. It follows from the above that the ramification type is an invariant under the $\pgl(2,\f_q)$-equivalence. Assume $\deg f=3$ and that $\f_q(X)/\f_q(f)$ is separable, equivalently, $\overline\f_q(X)/\overline\f_q(f)$ is separable. (Since $\deg f=3$, it is easy to see that $\f_q(X)/\f_q(f)$ is separable unless $p=\ch \f_q=3$ and $f=\phi^3$ for some $\phi\in\pgl(2,\f_q)$.) By the Hurwitz genus formula (\cite[Theorem~3.4.13]{Stichtenoth-2009}) applied to $\overline\f_q(X)/\overline\f_q(f)$, the different $\text{Diff}(\overline\f_q(X)/\overline\f_q(f))$ has degree 4. Using Dedekind's different theorem (\cite[Theorem~3.5.1]{Stichtenoth-2009}), we see that the possible ramification types of $f$ are given in Table~\ref{tb-ram}, where $*$ denotes the degree of the ramification point.

\begin{table}[H]
\vspace{-0.5em}
\caption{Possible ramification types of $f$ where $\deg f=3$ and $\f_q(X)/\f_q(f)$ is separable}
    \label{tb-ram}
    \centering
    {\renewcommand{\arraystretch}{1.5}
\begin{tabular}{c|c}
    $p$ & ramification type\\ 
    \hline 
    $2$ & $(2/*),\ (2/*,2/*),\ (2/*,3/*),\ (3/*,3/*)$ \\
    $3$ & $(2/*,2/*,2/*,2/*),\ (2/*,3/*),\ (3/*)$\\
    $\geq 5$ & $(2/*,2/*,2/*,2/*),\ (2/*,2/*,3/*),\ (3/*,3/*)$ 
\end{tabular}}
\end{table}


\section{Cross Ratio}\label{sec5}

Let $\f$ be any field. The rational function
\[
(X_1,X_2;X_3,X_4):=\frac{(X_1-X_3)(X_2-X_4)}{(X_1-X_4)(X_2-X_3)}\in\f(X_1,X_2,X_3,X_4)
\]
is called the {\em cross ratio} of $X_1,X_2,X_3,X_4$. To avoid confusion with generic quadruples, we write the cross ratio as
\[
C=C(X_1,X_2,X_3,X_4)=(X_1,X_2;X_3,X_4).
\]
The symmetric group $S_4$ acts on $\f(X_1,X_2,X_3,X_4)$ by permuting $X_1,\dots,X_4$: $\pi(f)=f(X_{\pi(1)},\dots,X_{\pi(4)})$ for $\pi\in S_4$ and $f\in\f(X_1,X_2,X_3,X_4)$. The stabilizer of $C$ in $S_4$ is
\[
V=\{(1),(1,2)(3,4),(1,3)(2,4),(1,4)(2,3)\}\vartriangleleft S_4,
\]
and we have $S_4=V\rtimes S_3$. For $\{i_1,i_2,i_3,i_4\}=\{1,2,3,4\}$, we write
\[
C_{i_1,i_2,i_3,i_4}=C(X_{i_1},X_{i_2},X_{i_3},X_{i_4}).
\]
Then the $S_4$-orbit of $C$ is 
\[
[C]=\bigl\{C_{i_1,i_2,i_3,4}:\{i_1,i_2,i_3\}=\{1,2,3\}\bigr\},
\]
where
\begin{equation}\label{C1234}
\begin{cases}
C_{1234}=C,\cr
C_{1324}=1-C,\vspace{0.2em}\cr
C_{2134}=\displaystyle \frac 1C,\vspace{0.2em}\cr
C_{2314}=\displaystyle 1-\frac 1C,\vspace{0.2em}\cr
C_{3124}=\displaystyle \frac 1{1-C},\vspace{0.2em}\cr
C_{3214}=\displaystyle \frac C{C-1}.
\end{cases}
\end{equation}
(In fact, fix any $i_j\in\{1,2,3,4\}$, $1\le j\le 4$, each element in $[C]$ is uniquely of the form $C_{i_1,i_2,i_3,i_4}$, where $(i_1,\dots,\widehat{i_j},\dots,i_4)$ is a permutation of $\{1,\dots,4\}\setminus\{i_j\}$.)

An important fact about the cross ratio is this: For $(a_1,\dots,a_4),(b_1,\dots,b_4)\in(\f\cup\{\infty\})^4$ with $|\{a_1,\dots,a_4\}|=4=|\{b_1,\dots,b_4\}|$, $C(a_1,\dots,a_4)=C(b_1,\dots,b_4)$ if and only if there exists $\phi\in\pgl(2,\f)$ such that $\phi(a_i)=b_i$ for all $1\le i\le 4$.


\section{Class III with $p\ge 5$}\label{sec6}

\subsection{Overview}\

Let $f\in\f_q(X)$ be in Class III, that is, $\deg f=3$, and there exists $\alpha\in\f_q\cup\{\infty\}$ such that $|f^{-1}(\alpha)|=3$ but $|f^{-1}(\beta)|\ne 2$ for all $\beta\in\f_q\cup\{\infty\}$. Let $p=\ch\f_q$. The assumption $p\ge 5$ is not needed until Section~\ref{sec6.4}.

Under equivalence, we may assume that $f^{-1}(\infty)=\{0,1,\infty\}$. Then it is easy to see that $f$ is equivalent to
\begin{equation}\label{fst}
f_{s,t}:=\frac{X^3+sX+t}{X(X-1)},\ \text{where}\ s,t\in\f_q,\ t\ne 0,\ 1+s+t\ne0.
\end{equation}
We have
\[
f_{s,t}'(X)=\frac{G_{s,t}(X)}{X^2(X-1)^2},
\]
where
\begin{equation}\label{gst}
G_{s,t}(X)=X^4-2X^3-sX^2-2tX+t.
\end{equation}
As seen in Section~\ref{sec4}, the ramification points of $f_{s,t}$ are precisely the roots of $G_{s,t}$, as clearly, $0,1,\infty$ are not ramification points of $f_{s,t}$. Moreover, among the roots of $G_{s,t}$, those that are conjugates over $\f_q$ have the same ramification index.

Before proceeding, we must realize that it is possible for $f_{s,t}$ to be in Class II. The next theorem tells exactly when this happens.

\begin{thm}\label{Theorem 2}
Let $s,t\in\f_q$ be such that $t\ne 0$ and $1+s+t\ne 0$. Then the following statements are equivalent.
\begin{itemize}
\item[(i)] $|f_{s,t}^{-1}(\alpha)|=2$ for some $\alpha\in\f_q$.

\item[(ii)] $G_{s,t}$ has a root $x\in\f_q$ with $x^3\ne -t$.

\item[(iii)] There exist $a,b\in\f_q$, $a\ne b$, such that 
\[
\begin{cases}
t=-a^2b,\cr
s=-2a-b+a^2+2ab.
\end{cases}
\]
\end{itemize}
\end{thm}

\begin{proof}
For $\alpha\in\f_q$, $|f_{s,t}^{-1}(\alpha)|=2$ if and only if 
\begin{equation}\label{ab}
X^3+sX+t-\alpha X(X-1)=(X-a)^2(X-b)
\end{equation}
for some $a,b\in\f_q$ with $a\ne b$. Comparing the coefficients in \eqref{ab} gives
\[
\begin{cases}
a^2b+t=0,\cr
-a^2-2ab+\alpha+s=0,\cr
2a+b-\alpha=0,
\end{cases}
\]
which is equivalent to
\begin{equation}\label{alpha-s-t}
\begin{cases}
\alpha=2a+b,\cr
t=-a^2b,\cr
s=-2a-b+a^2+2ab.
\end{cases}
\end{equation}
This proves (i) $\Leftrightarrow$ (iii).

In \eqref{alpha-s-t}, since $t\ne 0$, $a\ne b$ if and only if $a^3\ne -t$. Making the substitution $b=-t/a^2$ in the last equation of \eqref{alpha-s-t} gives
\begin{equation}\label{quartic-eq}
a^4-2a^3-sa^2-2ta+t=0.
\end{equation}
This proves (ii) $\Leftrightarrow$ (iii)
\end{proof}

The factorization of $G_{s,t}$ over $\f_q$ together with the information in Table~\ref{tb-ram} determines the ramification type of $f_{s,t}$ in Class III. Accordingly, we divide Class III into the following subclasses:

\begin{itemize}
\item[(III-a)] $G_{s,t}$ is irreducible over $\f_q$. In the case, $p>2$ and the ramification type of $f_{s,t}$ is $(2/4,2/4,2/4,2/4)$.
\medskip
\item[(III-b)] $G_{s,t}$ is a product of two different irreducible quadratics over $\f_q$. In the case, $p>2$ and the ramification type of $f_{s,t}$ is $(2/2,2/2,2/2,2/2)$.
\medskip
\item[(III-c)] $G_{s,t}$ is a square of an irreducible quadratic over $\f_q$. In the case, $p\ne 3$ and the ramification type of $f_{s,t}$ is $(3/2,3/2)$ for $p\ge 5$ and is $(2/2,2/2)$ or $(3/2,3/2)$ for $p=2$.
\medskip
\item[(III-d)] $G_{s,t}$ has at least one root in $\f_q$ and all roots $u$ of $G_{s,t}$ in $\f_q$ satisfy $u^3=-t$. In this case, we will see that if $p=2$, the ramification type of $f_{s,t}$ is $(3/1,3/1)$ (Section~\ref{sec6.7}); if $p=3$, the ramification type of $f_{s,t}$ is $(3/1)$ (Section~\ref{sec7}); if $p\ge 5$, the ramification type of $f_{s,t}$ is $(2/2,2/2,3/1)$ or $(3/1,3/1)$ (Section~\ref{sec6.7}).
\end{itemize}

\begin{lem}\label{lem:2/1}
Let $f\in\f_q(X)$ be of degree 3. Then $f$ is in Class II if and only if $f$ has a ramification point of type $2/1$.
\end{lem}

\begin{proof}
($\Rightarrow$) Under equivalence, we may assume $f^{-1}(0)=\{0,1\}$. Write $f=g/h$, where $g,h\in\f_q[X]$, $g$ is monic and $\gcd(g,h)=1$. Then $\deg g=3$. (Otherwise, $\infty\in f^{-1}(0)$.) It follows that $g(X)=X^2(X-1)$ or $X(X-1)^2$. Hence either 0 or 1 is a ramification point of $f$ of type $2/1$.

\medskip
($\Leftarrow$) Under equivalence, we may assume that 0 is a ramification point of type $2/1$, $f(0)=0$. Write $f=g/h$, where $g,h\in\f_q[X]$, $g$ monic and $\gcd(g,h)=1$. If $f(\infty)\ne 0$, then $\deg g=3$ and 
\[
g(X)=X^2(X-a),\quad a\in\f_q\setminus\{0\}.
\]
In this case, $f^{-1}(0)=\{0,a\}$. If $f(\infty)=0$, then $\deg g<3$ and hence $g(X)=X^2$. In this case, $f^{-1}(0)=\{0,\infty\}$. So $f$ is always in Class II.
\end{proof}

\subsection{Equivalence between $f_{s,t}$ and $f_{s',t'}$}\

For $s,t\in\f_q$ with $t(1+s+t)\ne 0$, we want to determine $\phi=\left[\begin{smallmatrix}a&b\cr c&d\end{smallmatrix}\right]\in\pgl(2,\f_2)$ such that $\psi\circ f_{s,t}\circ\phi=f_{s',t'}$ for some $\psi\in\pgl(2,\f_q)$ and $s',t'\in\f_q$, that is,
\begin{equation}\label{S}
\mathcal S(f_{s,t}\circ\phi)=\mathcal S(f_{s',t'}).
\end{equation}
Write
\[
f_{s,t}\circ\phi=\frac{A_3X^3+A_2X^2+A_1X+A_0}{B_3X^3+B_2X^2+B_1X+B_0},
\]
where
\[
\begin{cases}
A_0=b^3+b d^2 s+d^3 t,\cr
A_1=3 a b^2+2 b c d s+a d^2 s+3 c d^2 t,\cr
A_2=3 a^2 b + b c^2 s + 2 a c d s + 3 c^2 d t,\cr
A_3=a^3+a c^2 s+c^3 t,
\end{cases}
\]
\[
\begin{cases}
B_0=b^2 d - b d^2,\cr
B_1=b^2 c+2 a b d-2 b c d-a d^2,\cr
B_2=2 a b c-b c^2+a^2 d-2 a c d,\cr
B_3=a^2 c-a c^2,
\end{cases}
\]
and $(A_3,B_3)\ne (0,0)$.
We have
\begin{align*}
\mathcal S(f_{s,t}\circ\phi)&\ni B_3(A_3X^3+\cdots+A_0)-A_3(B_3X^3+\cdots+B_0)\cr
&=(bc-ad)\bigl(g_2X(X-1)+g_1X+g_0\bigr),
\end{align*}
where
\[
\begin{cases}
g_0=a^2 b^2-a b^2 c-a^2 b d-a b c d s-b c^2 d t-a c d^2 t+c^2 d^2 t,\cr
g_1=a^4+2 a^3 b-2 a^3 c-3 a^2 b c-a^3 d-a^2 c^2 s-a b c^2 s-a^2 c d s-2 a c^3 t-b c^3 t+c^4 t-3 a c^2 d t+2 c^3 d t,\cr
g_2=a^4 - 2 a^3 c - a^2 c^2 s - 2 a c^3 t + c^4 t.
\end{cases}
\]
Therefore, \eqref{S} holds for some $(s',t')$ if and only if
\begin{equation}\label{g0g1}
\begin{cases}
g_0=0,\cr
g_1=0.
\end{cases}
\end{equation}

The set-wise stabilizer of $\{0,1,\infty\}$ in $\pgl(2,\f_q)$ is an isomorphic copy of $S_3$:
\[
S_3=\Bigl\{
\left[\begin{matrix}1&0\cr 0&1\end{matrix}
\right],
\left[\begin{matrix}1&-1\cr 0&-1\end{matrix}
\right],
\left[\begin{matrix}0&1\cr 1&0\end{matrix}
\right],
\left[\begin{matrix}0&1\cr -1&1\end{matrix}
\right],
\left[\begin{matrix}1&0\cr 1&-1\end{matrix}
\right],
\left[\begin{matrix}1&-1\cr 1&0\end{matrix}
\right]
\Bigr\}<\pgl(2,\f_q).
\]
When $\phi\in S_3$, \eqref{g0g1} is satisfied, and hence \eqref{S} holds for some $(s',t')$, i.e., $\psi\circ f_{s,t}\circ\phi=f_{s',t'}$ for some $\psi\in\pgl(2,\f_q)$. The pair $(s',t')$ is given in Table~\ref{tb2}, where $o(\phi)$ denotes the order of $\phi$.

\begin{table}[ht]
\caption{$S_3$ acting on $f_{s,t}$: $\psi\circ f_{s,t}\circ\phi=f_{s',t'}$, $\phi\in S_3$, $\psi\in\pgl(2,\f_q)$}\label{tb2}
    \centering
\begin{tabular}{c|c|l}
$\phi$ & $o(\phi)$ &\hfil $(s',\ t')$\\ \hline 
\vspace{-0.8em} & \\
$\left[\begin{matrix}1&0\cr 0&1\end{matrix}
\right]$ & 1& $(s,\ t)$\\ 
\vspace{-0.8em} & \\
$\left[\begin{matrix}1&-1\cr 0&-1\end{matrix}
\right]$ & 2 & $(s,\ -1-s-t)$\\
\vspace{-0.8em} & \\
$\left[\begin{matrix}0&1\cr 1&0\end{matrix}
\right]$ & 2 & $\displaystyle \Bigl(\frac st,\ \frac 1t\Bigr)$\\
\vspace{-0.8em} & \\
$\left[\begin{matrix}0&1\cr -1&1\end{matrix}
\right]$ &  3 &$\displaystyle \Bigl(\frac st,\ -\frac{1+s+t}t\Bigr)$\\
\vspace{-0.8em} & \\
$\left[\begin{matrix}1&0\cr 1&-1\end{matrix}
\right]$ & 2 &  $\displaystyle \Bigl(-\frac s{1+s+t},\ -\frac t{1+s+t}\Bigr)$\\
\vspace{-0.8em} & \\
$\left[\begin{matrix}1&-1\cr 1&0\end{matrix}
\right]$ & 3 & $\displaystyle \Bigl(-\frac s{1+s+t},\ -\frac 1{1+s+t}\Bigr)$
\end{tabular}
\end{table}

It is easy to see that under condition~\eqref{g0g1}, 
\begin{equation}\label{notinS3}
\phi\in S_3\ \Leftrightarrow\ ac(a-c)=0\  \Leftrightarrow\ cd(c+d)=0.
\end{equation}
When $ac(a-c)\ne 0$ and $cd(c+d)\ne 0$, solving \eqref{g0g1} for $s,t$ in terms of $a,b,c,d$ gives 
\begin{equation}\label{st}
\begin{cases}
s=\displaystyle \frac{a^2 d+2 a b c+2 a b d-2 a c d-a d^2+b^2 c-b c^2-2 b c d}{c d (c+d)}, \vspace{0.4em}\cr
t=\displaystyle -\frac{a b (a+b)}{c d (c+d)}.
\end{cases}
\end{equation}
In the above, replacing $\phi$ with $\phi^{-1}=\left[\begin{smallmatrix}d&-b\cr -c&a\end{smallmatrix}\right]$ gives
\begin{equation}\label{s'=,t'=}
\begin{cases}
s'=\displaystyle \frac{a^2 d+2 a b c+2 a b d-2 a c d-a d^2+b^2 c-b c^2-2 b c d}{a c (a-c)}, \vspace{0.4em}\cr
t'=\displaystyle \frac{b d (b-d)}{a c (a-c)}.
\end{cases}
\end{equation}

To summarize, we have the following theorem:

\begin{thm}\label{summary}
Equation~\eqref{S} holds if and only if
\begin{itemize}
\item[(i)] $\phi\in S_3$ and $(s',t')$ is given by Table~\ref{tb2}, or
\item[(ii)] $\phi\notin S_3$ and \eqref{st}, \eqref{s'=,t'=} are satisfied.
\end{itemize}
\end{thm}

\subsection{The invariant $\theta$}\

\begin{thm}\label{thm:(-3,1)}
Let $s,t\in\overline\f_q$ with $t(1+s+t)\ne0$. Then $f_{s,t}$ has no ramification point of index 2 if and only if $(s,t)=(-3,1)$.
\end{thm}

\begin{proof}
First note that Theorem~\ref{Theorem 2} and Lemma~\ref{lem:2/1} still hold with $\f_q$ replaced by $\overline\f_q$. 

\medskip
($\Leftarrow$) We have 
\[
G_{-3,1}(X)=X^4-2X^3+3X^2-2X+1=(X^2-X+1)^2=(X+\epsilon)^2(X+\epsilon^{-1})^2,
\]
where $\epsilon\in\overline\f_q$ is such that $\epsilon^3=1$. By Theorem~\ref{Theorem 2} (with $\f_q$ replaced by $\overline\f_q$), $|f_{-3,1}^{-1}(\alpha)|\ne 2$ for all $\alpha\in\Bbb P^1(\overline\f_q)$. By Lemma~\ref{lem:2/1} with $\f_q$ replaced by $\overline\f_q$, $f_{-3,1}$ has no ramification point of index 2.

\medskip
($\Rightarrow$) Let $a\in\Bbb P^1(\overline\f_q)$ be a ramification point of $f_{s,t}$. By assumption, its ramification index is 3. Note that $a\notin\{0,1,\infty\}$, as it is easy to see that $0,1,\infty$ are not ramification points of $f_{s,t}$. Let $b=f_{s,t}(a)\in\overline\f_q$. Then 
\[
X^3+sX+t-bX(X-1)=(X-a)^3=X^3-3aX^2+3a^2X-a^3.
\]
Comparing the coefficients gives
\[
\begin{cases}
b=3a,\cr
t=-a^3,\cr
s=3(a^2-a).
\end{cases}
\]
Since $t(1+s+t)\ne 0$, we have $a\ne 0,1$. Now
\[
G_{s,t}(X)=X^4-2X^3-3(a^2-a)X^2+2a^3X-a^3=(X-a)^2\bigl(X^2+2(a-1)X-a\bigr).
\]
By Lemma~\ref{lem:2/1} and Theorem~\ref{Theorem 2} (with $\f_q$ replaced by $\overline\f_q$), all roots of $x$ of $G_{s,t}$ satisfy $x^3=-t=a^3$. Hence
\[
X^2+2(a-1)X-a=(X-\epsilon^i a)(x-\epsilon^j a)
\]
where $\epsilon^3=1$ and $i,j\in\Bbb Z$. Comparing the coefficients gives
\[
\begin{cases}
-a=\epsilon^{i+j}a^2,\cr
2(a-1)=-(\epsilon^i+\epsilon^j)a.
\end{cases}
\]
Hence $a=-\epsilon^{2(i+j)}$, whence $t=-a^3=1$. If $o(\epsilon^{i+j})=3$, then $s=3(a^2-a)=-3$. If $\epsilon^{i+j}=1$, then $a=-1$ and $p\ne 2$. Then $2(a-1)=-(\epsilon^i+\epsilon^j)a$ becomes $4=-(\epsilon^i+\epsilon{-i})=-2$ or 1. Hence $p=3$ so $s=0$.
\end{proof}

\begin{cor}\label{cor:(-3,1)}
If $f_{-3,1}\sim f_{s,t}$ over $\overline\f_q$, then $(s,t)=(-3,1)$.
\end{cor}

\begin{proof}
Since $f_{s,t}$ also has no ramification point of index 2, the conclusion follows from Theorem~\ref{thm:(-3,1)}.
\end{proof}

Let $f\in\overline\f_q(X)$ be of degree 3 and have a ramification point of index 2. By Theorem~\ref{Theorem 1} (with $\f_q$ replaced by $\overline\f_q$), $f$ is equivalent (over $\overline\f_q$) to precisely one of $(X^3+1)/X$ and $(X^3+X^2+t)/X$, where $0\ne t\in\overline\f_q$. Define
\[
\theta(f)=
\begin{cases}
0&\text{if}\ f\sim\displaystyle\frac{X^3+1}X,\vspace{0.4em}\cr
\displaystyle\frac1t&\text{if}\ f\sim\displaystyle\frac{X^3+X^2+t}X.
\end{cases}
\]
Then $\theta(f)$ is an invariant under the equivalence over $\overline\f_q$. Moreover, if $f\in\f_q(X)$, then $\theta(f)\in\f_q$. To see this claim, let $\sigma:\overline\f_q\to\overline\f_q$, $a\mapsto a^q$ be the Frobenius map and extend it to $\overline\f_q[X]\to\overline\f_q[X]$. Assume that $f\sim g$, where $g=(X^3+1)/X$ or $(X^3+X^2+t)/X$, $0\ne t\in\overline\f_q$. Then
\[
f=\sigma(f)\sim\sigma(g)=\Bigl(\frac{X^3+1}X\ \text{or}\ \frac{X^3+X^2+\sigma(t)}X\Bigr).
\]
It follows that 
\[
\theta(f)=(0\ \text{or}\ 1/\sigma(t))=(\sigma(0)\ \text{or}\ \sigma(1/t))=\sigma(\theta(f)).
\]
Hence $\theta(f)\in\f_q$.

Given $(s,t)\in\f_q^2$ with $t(1+s+t)\ne 0$ and $(s,t)\ne(-3,1)$, by Theorem~\ref{thm:(-3,1)}, $f_{s,t}$ has a ramification point of index 2, hence $\theta(f_{s,t})$ is defined. We now compute $\theta(f_{s,t})$ explicitly. Let $a\in\overline\f_q$ be a ramification point of $f_{st}$ of index 2, and let $\alpha=f(a)$. Clearly, $0,1,\infty$ are not ramification points of $f_{s,t}$, so $a\notin\{0,1,\infty\}$ and hence $\alpha\ne\infty$. Then we have
\[
X^3+sX+t-\alpha X(X-1)=(X-a)^2(X-b),
\]
where $\in\overline\f_q$, $b\ne a$. (This equation is the same as \eqref{ab}, but here $a,b,\alpha$ are not necessarily in $\f_q$.) In what follows, $\sim$ means equivalence over $\overline\f_q$.
We have
\begin{align*}
f_{s,t}\,&\sim\frac{X^3+sX+t}{X(X-1)}-\alpha\cr
&=\frac{(X-a)^2(X-b)}{X(X-1)}\kern8em (X\mapsto X+a)\cr
&\sim\frac{X^2(X+a-b)}{(X+a)(X+a-1)}\kern6.8em (X\mapsto 1/X)\cr
&\sim\frac{a(a-1)X^3+(2a-1)X^2+X}{(a-b)X+1}\kern2em (X\mapsto X-1/(a-b))\cr
&\sim\frac{c_3X^3+c_2X^2+c_0}X,
\end{align*}
where
\[
\begin{cases}
c_0=-b(b-1),\cr
c_2=-(a-b)^2 (-2 a+a^2-b+2 a b),\cr
c_3=(a-1) a (a-b)^3.
\end{cases}
\]
By \eqref{alpha-s-t},
\[
c_2=-(a-b)^2s.
\]
If $s\ne 0$, then $c_2\ne 0$. Thus, with $X\mapsto (c_2/c_3)X$, we have
\[
f_{s,t}\sim\frac{X^3+X^2+t'}X,
\]
where
\[
t'=\frac{c_0c_3^2}{c_2^3}=\frac{b(b-1)(a-1)^2a^2}{s^3}=\frac{t(a^2+t)(a-1)^2}{a^2s^3}\kern3em\text{(by \eqref{alpha-s-t})}.
\]
By \eqref{quartic-eq},
\[
(a^2+t)(a-1)^2=a^2(1+s+t),
\]
so
\[
t'=\frac{t(1+s+t)}{s^3}.
\]
Therefore,
\begin{equation}\label{theta-fst}
\theta(f_{s,t})=\frac{s^3}{t(1+s+t)}.
\end{equation}
If $s=0$, then $c_2=0$. In this case, $f_{s,t}\sim(X^3+1)/X$ and $\theta(f_{s,t})=0$, so \eqref{theta-fst} also holds.  

We write 
\begin{equation}\label{eq:theta}
\theta(s,t)=\frac{s^3}{t(1+s+t)},
\end{equation}
for all $(s,t)\in{\overline\f}^2$ with $t(1+s+t)\ne0$, including $(s,t)=(-3,1)$. Then we know that if $f_{s,t}\sim f_{s',t'}$, then $\theta(s,t)=\theta(s',t')$.

\begin{rmk}\rm
When $f_{s,t}\sim f_{s',t'}$, one can also verify that $\theta(s,t)=\theta(s',t')$ through direct computation. Let $\phi=\left[\begin{smallmatrix}a&b\cr c&d\end{smallmatrix}\right]\in\pgl(2,\overline\f_q)$ be such that $\mathcal S(f_{s,t}\circ\phi)=\mathcal S(f_{s',t'})$. If $\phi\in S_3$, one uses Table~\ref{tb2}; if $\phi\notin S_3$, one uses \eqref{st} and \eqref{s'=,t'=}.
\end{rmk}

Let $\Omega=\{(s,t)\in\f_q^2:t(1+s+t)\ne 0\}$, and define
\begin{align*}
\Theta_{\rm a}=\,&\{\theta(s,t): (s,t)\in\Omega,\ \text{$f_{s,t}$ is in Class III-a}\}\cr
=\,&\{\theta(s,t): (s,t)\in\Omega,\ \text{$G_{s,t}$ is irreducible over $\f_q$}\},
\end{align*}
\begin{align*}
\Theta_{\rm b}=\,&\{\theta(s,t): (s,t)\in\Omega,\ \text{$f_{s,t}$ is in Class III-b}\}\cr
=\,&\{\theta(s,t): (s,t)\in\Omega,\ \text{$G_{s,t}$ is a product of two different}\cr
&\text{irreducible quadratics over $\f_q$}\},
\end{align*}
\begin{align*}
\Theta_{\rm c}=\,&\{\theta(s,t): (s,t)\in\Omega,\ \text{$f_{s,t}$ is in Class III-c}\}\cr
=\,&\{\theta(s,t): (s,t)\in\Omega,\ \text{$G_{s,t}$ is a square of an irreducible quadratic over $\f_q$}\}.
\end{align*}
We will determine $\Theta_a,\Theta_b$ and $\Theta_c$ in the next three subsections for $p\ge 5$.

For polynomials $f(X),g(X)$ over a field, we denote the discriminant of $f$ by $D(f)$ and the resultant of $f$ and $g$ by $\res(f,g)$ or $\res(f,g;X)$.

\begin{lem}\label{lpq}
Let $F$ be a field and $f,g\in F[X]\setminus F$ be monic polynomials. Then
\[
D(fg)=D(f)D(g)\,\text{\rm Res}(f,g)^2.
\]
\end{lem}

\begin{proof} Write $f=\prod_{i=1}^m(X-u_i)$ and $g=\prod_{j=1}^n(X-v_j)$, where $u_i,v_j\in\overline F$. We have 
\[
D(fg)=D(f)D(g)\Bigl(\prod_{1\le i\le m,\, 1\le j\le n}(u_i-v_j)\Bigr)^2,
\]
where
\[
\prod_{1\le i\le m,\, 1\le j\le n}(u_i-v_j)=\res(f,g).
\]
\end{proof}

\begin{lem}\label{disc}
Let $q$ be odd and $f\in\f_q[X]$ be irreducible of degree $n$. Then
\[
D(f)\ \text{\rm is}\begin{cases}
\text{\rm a nonzero square in $\f_q$} &\text{\rm if $n$ is odd},\cr
\text{\rm a nonsquare in $\f_q$} &\text{\rm if $n$ is even}.
\end{cases}
\]
\end{lem}

\begin{proof}
The Galois group of $f$ over $\f_q$ is cyclic of order $n$, which is contained in $A_n$ if and only if $n$ is odd. Hence the conclusion.
\end{proof}

\begin{prop}\label{P-Q}
Let $f\in\f_q[X]$ be a quartic polynomial.
\begin{itemize}
\item[(i)] $D(f)$ is a nonsquare in $\f_q$ if and only if  $f$ is irreducible over $\f_q$ or is a product of two distinct linear polynomials and an irreducible quadratic over $\f_q$.

\item[(ii)] $D(f)$ is a nonzero square in $\f_q$ if and only if $f$ is a product of two distinct irreducible quadratics over $\f_q$ or a product of a linear polynomial and an irreducible cubic or a product of four distinct linear polynomials over $\f_q$.

\item[(iii)] $D(f)=0$ if and only if $f$ not separable.
\end{itemize}
\end{prop}

\begin{proof}
The conclusions follow from the above two lemmas.
\end{proof}

For $(s,t)\in\Omega$, we have
\begin{align}\label{Dfst}
D(f_{s,t})\,& =16 t (1+s+t) (s^3-27 t-27 s t-27 t^2)\\
&=16 t^2(1+s+t)^2\Bigl(\frac{s^3}{t(1+s+t)}-27\Bigr)\cr
&=16 t^2(1+s+t)^2(\theta(s,t)-27).\nonumber
\end{align}

\begin{cor}\label{C-tha}
We have
\begin{equation}\label{tha1}
\Theta_{\rm a}\subset\{27+a: \text{\rm $a$ is a nonsquare in $\f_q$}\},
\end{equation}
\begin{equation}\label{tha2}
\Theta_{\rm b}\subset\{27+a: \text{\rm $a$ is a nonzero square in $\f_q$}\},
\end{equation}
\begin{equation}\label{tha3}
\Theta_{\rm c}\subset\{27\}.
\end{equation}
\end{cor}

\begin{proof}
This follows from Proposition~\ref{P-Q} and \eqref{Dfst}.
\end{proof}

For the rest of Section~\ref{sec6}, we assume $p=\ch\f_q\ge 5$.

\subsection{Class III-a}\label{sec6.4}\

For each $F(X,Y)\in\f_q[X,Y]$, we write $V(F)=\{(x,y)\in\f_q^2:F(x,y)=0\}$.
The main result of this subsection is the following theorem:

\begin{thm}\label{C-tha1}
Assume $p\ge 5$. Then 
\begin{equation}\label{thaa}
\Theta_{\rm a}=\{27+a: \text{\rm $a$ is a nonsquare in $\f_q$}\},
\end{equation}
and $|\Theta_a|=(q-1)/2$.
\end{thm}

\begin{proof}
Let $a$ be a nonsquare in $\f_q$ and let $b=27+a$. We want to show that $b\in\Theta_{\rm a}$.

First assume $b=0$. (This happens only when $a=-27$ is a nonsquare in $\f_q$.) To show that $0\in\Theta_{\rm a}$, we show that there exists $t\in\f_q\setminus\{0,-1\}$ such that $G_{0,t}$ is irreducible over $\f_q$. Assume to the contrary that for all $t\in\f_q\setminus\{0,-1\}$, $G_{0,t}$ is not irreducible over $\f_q$. Since $D(G_{0,t})=16t^2(1+t)^2a$ is a nonsquare in $\f_q$, by Proposition~\ref{P-Q} (i), $G_{0,t}$ has a root $x_t\in\f_q$. It follows that
\[
t=\frac{x_t^3(x_t-2)}{2x_t-1},\quad x_t\ne 1/2,\, 0,\, 2.
\]
This gives at most $q-3$ values of $t$, which is a contradiction.

Now assume $b\ne 0$. Assume to the contrary that $b\notin\Theta_{\rm a}$. We will arrive at a contradiction through several steps.

\medskip

{\bf Step 1.} The polynomials $F_b$ and $H_b$.\

Define
\begin{equation}\label{Fb}
F_b(S,T)=S^3-bT(1+S+T)\in\f_q[S,T].
\end{equation}
Let $(s,t)\in V(F_b)\setminus\{s=0\}$. We have
\begin{equation}\label{=b}
\theta(s,t)=\frac{s^3}{t(1+s+t)}=b.
\end{equation}
Since $b\notin\Theta_{\rm a}$, $G_{s,t}$ is not irreducible over $\f_q$. Since $D(f_{s,t})$ is a nonsquare in $\f_q$, by Proposition~\ref{P-Q} (i), $G_{s,t}$ is a product of two distinct linear polynomials and an irreducible quadratic over $\f_q$. Let $x,y$ be the roots of $G_{s,t}$ in $\f_q$. Then 
\begin{equation}\label{sys}
\begin{cases}
x^2s+(2x-1)t=x^4-2x^3,\cr
y^2s+(2y-1)t=y^4-2y^3.
\end{cases}
\end{equation}
In the above, we have $x\ne 0$ and $y\ne0$. (Otherwise, $t=0$ and it follows from $F_b(s,t)=0$ that $s=0$, which is a contradiction.) We claim that
\begin{equation}\label{det}
\det\left[\begin{matrix} x^2&2x-1\cr y^2& 2y-1\end{matrix}\right]=(x - y) (-x - y + 2 x y)\ne 0.
\end{equation}
Otherwise, since \eqref{sys} has a solution for $(s,t)$, we have
\[
\det\left[\begin{matrix} x^2&x^4-2x^3\cr y^2& y^4-2y^3\end{matrix}\right]=-x^2 y^2 (x - y) (-2 + x + y)= 0.
\]
Thus $-2 + x + y=0$. Solving
\[
\begin{cases}
-2 + x + y=0,\cr
-x - y + 2 x y=0
\end{cases}
\]
gives $(x,y)=(1,1)$, which is a contradiction since $x\ne y$. Therefore \eqref{det} is proved. Now \eqref{sys} has a unique solution
\begin{equation}\label{st=}
\begin{cases}
s=\displaystyle\frac{2 x^2 - x^3 + 2 x y - 5 x^2 y + 2 x^3 y + 2 y^2 - 5 x y^2 + 
 2 x^2 y^2 - y^3 + 2 x y^3}{-x-y+2 x y},\vspace{0.4em}\cr
t=-\displaystyle\frac{x^2 y^2 (-2 + x + y)}{-x-y+2 x y}.
\end{cases}
\end{equation}
Making the substitutions \eqref{st=} in \eqref{=b} gives 
\[
0=F_b(s,t)=\frac{H_b(x,y)}{(-x-y+2xy)^3},
\]
where
\begin{align}\label{Hb}
&H_b(X,Y)=\\
&-b (-1+X)^2 X^2 (-1+Y)^2 Y^2 (-2+X+Y) (X+Y) (-X-Y+2 X Y)\cr
&+(2 X^2-X^3+2 X Y-5 X^2 Y+2 X^3 Y+2 Y^2-5 X Y^2+2 X^2 Y^2-Y^3+2 X Y^3)^3.\nonumber
\end{align}
Therefore $H_b(x,y)=0$. 

In \eqref{Hb}, write
\[
N(X,Y)=2 X^2-X^3+2 X Y-5 X^2 Y+2 X^3 Y+2 Y^2-5 X Y^2+2 X^2 Y^2-Y^3+2 X Y^3,
\]
and note that the numerator of $s$ in \eqref{st=} is $N(x,y)$. For $(x,y)\in V(H_b)$, $N(x,y)\ne0$ implies $-x-y+2 x y\ne0$. Let
\[
\Gamma=\{(x,y)\in V(H_b):  N(x,y)=0\}.
\]
Define
\begin{equation}\label{eq:Phi}
\begin{array}{cccl}
\Phi:&V(H_b)\setminus\Gamma&\longrightarrow&V(F_b)\setminus\{s=0\}=V(F_b)\setminus\{(0,0),(0,-1)\}\cr
&(x,y)&\xmapsto{\hspace{1em}} &(s,t),
\end{array}
\end{equation}
where $s,t$ are given by \eqref{st=}. We have shown that $\Phi$ is an onto map. The system
\[
\begin{cases}
H_b(x,y)=0,\cr
N(x,y)=0
\end{cases}
\]
is equivalent to 
\[
\begin{cases}
N(x,y)=0,\cr
x y(-1+x)(-1+y)(-2+x+y)(x+y)(-x-y+2 x y)=0,
\end{cases}
\]
whose solution set is $\{(0,0),(1,1),(0,2),(2,0),(1,-1),(-1,1)\}$. Therefore,
\[
\Gamma=\{(0,0),(1,1),(0,2),(2,0),(1,-1),(-1,1)\}.
\]
For each $(s,t)\in V(F_b)\setminus\{s=0\}$, $|\Phi^{-1}(s,t)|\ge 2$ since $(x,y)\in\Phi^{-1}(s,t)$ implies $x\ne y$ and $(y,x)\in\Phi^{-1}(s,t)$. Therefore by \eqref{eq:Phi},
\[
|V(F_b)|-2\le \frac 12(|V(H_b)|-|\Gamma|)=\frac 12|V(H_b)|-3,
\]
i.e., 
\begin{equation}\label{VFb}
|V(F_b)|\le\frac12|V(H_b)|-1.
\end{equation}

\medskip

{\bf Step 2.} Estimate for $|V(F_b)|$.\

\medskip
First, $F_b(S,T)=-bT^2-b(1+S)T+S^3\in \overline\f_q[S,T]$ is irreducible since 
$(-b(1+S))^2+4bS^3$ is not a square in $\overline\f_q[S]$. Then by the Hasse-Weil bound,
\begin{equation}\label{vfb}
|V(F_b)|\ge q+1-(3-1)(3-2)q^{1/2}-1=q-2q^{1/2}.
\end{equation}
(The term $-1$ represents the point at $\infty$ on the curve $F_b(S,T)=0$.)

\medskip

{\bf Step 3.} Factorization of $H_b$.\

Let 
\[
\alpha(X)=8-(15+a)X+6X^2+X^3\in\f_q[X].
\]
Since 
\[
D(\alpha)=4a(27+a)^2
\]
is nonsquare in $\f_q$, by \cite{Dickson-BAMS-1906}, $\alpha$ is a product of a linear polynomial and an irreducible quadratic polynomial in $\f_q[X]$. (This fact also follows from Lemma~\ref{disc}.)

Let $s_1=X+Y$ and $s_2=XY$. We have
\[
H_b(X,Y)=K(s_1,s_2),
\]
where
\begin{dmath*}
K(X,Y)=8 X^6-12 X^7+6 X^8-X^9-24 X^4 Y+42 X^6 Y-30 X^7 Y+6 X^8 Y-30 X^2 Y^2-2 a X^2 Y^2+171 X^3 Y^2+5 a X^3 Y^2-180 X^4 Y^2-4 a X^4 Y^2+15 X^5 Y^2+a X^5 Y^2+42 X^6 Y^2-12 X^7 Y^2-8 Y^3+84 X Y^3+4 a X Y^3-330 X^2 Y^3-14 a X^2 Y^3+442 X^3 Y^3+14 a X^3 Y^3-180 X^4 Y^3-4 a X^4 Y^3+8 X^6 Y^3-24 Y^4+168 X Y^4+8 a X Y^4-330 X^2 Y^4-14 a X^2 Y^4+171 X^3 Y^4+5 a X^3 Y^4-24 X^4 Y^4-24 Y^5+84 X Y^5+4 a X Y^5-30 X^2 Y^5-2 a X^2 Y^5-8 Y^6.
\end{dmath*}
Let $c$ be any root of $\alpha(X)$. We found that
\begin{equation}\label{fac}
K(X,Y)=P_c\cdot (B_2+\cdots+B_6),
\end{equation}
where
\begin{equation}\label{pc}
P_c=cY+2X^2-(4+c)XY+cY^2-X^2(X-2Y),
\end{equation}
\[
B_2=-\frac 8cY^2,
\]
\[
B_3=-\frac 2cY\bigl[c(c+6)X^2-2(4+6c+c^2)XY+8Y^2 \bigr],
\]
\[
B_4=\frac 1c\bigl[4cX^4+c(2+3c)X^3Y-8(1+4c+c^2)X^2Y^2+4(4+6c+c^2)XY^3-8Y^4\bigr],
\]
\[
B_5=-X^2(X-2Y)\bigl[4X^2+(c-2)XY-(c+6)Y^2\bigr],
\]
\[
B_6=X^4(X-2Y)^2.
\]
The factorization~\eqref{fac} was found through computer experiment. The process is difficult to describe here but the result is easily verified; see Appendix~\ref{app-b}.

Let the roots of $\alpha$ be $c_1,c_2,c_3$, where $c_1\in\f_q$, $c_2,c_3\in\f_{q^2}\setminus\f_q$, $c_3=c_2^q$. Clearly, $P_{c_1}, P_{c_1}, P_{c_3}$ are pairwise relatively prime, hence $P_{c_1}P_{c_2}P_{c_3}\mid K$. Comparing the coefficients of $Y^6$ and noting that $c_1c_2c_3=-8$, we get
\[
K=P_{c_1}P_{c_2}P_{c_3}.
\]
Therefore,
\begin{equation}\label{c1c2c3}
H_b(X,Y)=P_{c_1}(s_1,s_2)P_{c_2}(s_1,s_2)P_{c_3}(s_1,s_3).
\end{equation}

\medskip

{\bf Step 4.} Determination of $V(P_{c_2}(s_1,s_2)P_{c_3}(s_1,s_3))$.\

We write \eqref{pc} as 
\[
P_c(s_1,s_2)=(-2+X+Y)(X+Y)(-X-Y+2XY)+cXY(-1+X)(-1+Y).
\]
In the above, set $c=c_1,c_2$, and note that $c_1,c_2\in\f_{q^2}\setminus\f_q$. Therefore, $(x,y)\in V(P_{c_2}(s_1,s_2)P_{c_3}(s_1,s_3))$ if and only if
\[
\begin{cases}
(-2+x+y)(x+y)(-x-y+2xy)=0,\cr
xy(-1+x)(-1+y)=0.
\end{cases}
\]
The solution set of the above system is 
\[
\Gamma=\{(0,0),(1,1),(0,2),(2,0),(1,-1),(-1,1)\}.
\]
Therefore,
\begin{equation}\label{=6}
|V(P_{c_2}(s_1,s_2)P_{c_3}(s_1,s_3))|=6.
\end{equation}
Note: For our purpose, we only need a small constant upper bound for $|V(P_{c_2}(s_1,s_2)P_{c_3}(s_1,s_3))|$, which can be obtained using Bezout's theorem. However, in this case, an explicit determination is easier.

\medskip

{\bf Step 5.} Absolute irreducibility of $P_{c_1}(s_1,s_2)$.\

To simplify the notation, we write $c=c_1$, which is the root of $\alpha(X)$ in $\f_q$. 

We first show that $P_c(X,Y)$ is absolutely irreducible. Assume the contrary. Write \eqref{pc} as
\[
P_c(X,Y)=cY^2+(c-(4+c)X+2X^2)Y+X^2(2-X).
\]
Then 
\[
(c-(4+c)X+2X^2)^2-4c X^2(2-X)
\]
is a square in $\overline\f_q[X]$, say
\[
(c-(4+c)X+2X^2)^2-4c X^2(2-X)-(2X^2+dX+c)^2=0,
\]
where $d\in\overline\f_q$. In the above,
\[
\text{LHS}=X\bigl(-2c(4+c+d)+(16+c^2-d^2)X-4(4+d)X^2 \bigr).
\]
It follows that $4+c+d=0=4+d$, which is impossible.

Now we show that $P_c(X+Y,XY)$ is absolutely irreducible. Assume the contrary. Since $P_c(X,Y)$ is absolutely irreducible, by \cite[Proposition~1.7]{Sze-Thesis-2023}, we have
\begin{equation}\label{LL}
P_c(X+Y,XY)=L(X,Y)L(Y,X),
\end{equation}
where 
\[
L(X,Y)=a_{10}X+a_{01}Y+a_{20}X^2+a_{11}XY+a_{02}Y^2\in\overline\f_q[X,Y].
\]
Write
\[
P_c(X+Y,XY)-L(X,Y)L(Y,X)=\sum_{i,j}c_{ij}X^iY^j,
\]
where
\begin{align*}
c_{20}\,&=2-a_{01} a_{10},\\
c_{30}\,&=-1-a_{02} a_{10}-a_{01} a_{20},\\
c_{40}\,&=-a_{02} a_{20},\\
c_{11}\,&=4 - a_{01}^2 - a_{10}^2 + c,\\
c_{21}\,&=-7 - a_{01} a_{02} - a_{01} a_{11} - a_{10} a_{11} - a_{10} a_{20} - c,\\
c_{31}\,&=2 - a_{02} a_{11} - a_{11} a_{20},\\
c_{02}\,&=c_{20},\\
c_{12}\,&=c_{21},\\
c_{22}\,&=4 - a_{02}^2 - a_{11}^2 - a_{20}^2 + c,\\
c_{03}\,&=c_{30},\\
c_{13}\,&=c_{31},\\
c_{04}\,&=c_{40}.
\end{align*}
Thus $a_{02}a_{20}=0$. Without loss of generality, we may assume $a_{20}=0$. Now
\[
\begin{cases}
c_{20}=2-a_{01}a_{10},\cr
c_{30}=-1-a_{02}a_{10},
\end{cases}
\]
and hence
\[
a_{01}=\frac 2{a_{10}},\quad a_{02}=\frac{-1}{a_{10}}.
\]
Now we have
\[
c_{31}=\frac{2a_{10}+a_{11}}{a_{10}},
\]
whence
\[
a_{11}=-2a_{10}.
\]
It follows that 
\[
c_{21}=\frac 1{a_{10}^2}\bigl(2-(3+c)a_{10}^2+2a_{10}^4\bigr),
\]
\[
c_{22}=\frac{-1}{a_{10}^2}\bigl(1-(4+c)a_{10}^2+4a_{10}^4\bigr).
\]
Therefore,
\[
0=\text{Res}\bigl(2-(3+c)a_{10}^2+2a_{10}^4,\,1-(4+c)a_{10}^2+4a_{10}^4;a_{10}\bigr)=4(-1+c)^2(8+c)^2.
\]
However, from $\alpha(c)=0$, we have
\[
a=\frac{(-1+c)^2(8+c)}c.
\]
We have a contradiction since $a\ne 0$.

\medskip

{\bf Step 6.} Estimate for $|V(P_{c_1}(s_1,s_2))|$.\

Since $P_{c_1}(s_1(X,Y),s_2(X,Y))$ is absolutely irreducible and has degree $4$, by the Hasse-Weil bound, 
\begin{equation}\label{vpc}
|V(P_{c_1}(s_1(X,Y),s_2(X,Y)))|\le q+1+(4-1)(4-2)q^{1/2}-2=q+6q^{1/2}-1.
\end{equation}
(In the above, $4=\deg P_{c_1}(X+Y,XY)$ and $2$ is a lower bound for the number of rational points at $\infty$ on the curve $P_{c_1}(X+Y,XY)=0$.)

\medskip

{\bf Step 7.} Completion of the proof.\

Now by \eqref{c1c2c3}, \eqref{=6} and \eqref{vpc}, we have
\[
|V(H_b)|\le |V(P_{c_1}(s_1,s_2))|+|V(P_{c_2}(s_1,s_2)P_{c_3}(s_1,s_2))|\le q+6q^{1/2}+5.
\]
When $q\ge 107$, by \eqref{vfb}, we have
\[
|V(F_b)|\ge q-2q^{1/2}>\frac 12\bigl(q+6q^{1/2}+5\bigr)-1\ge\frac 12|V(H_b)|-1,
\]
which is a contradiction to \eqref{VFb}. When $q\le 103$, Theorem~\ref{C-tha1} is verified by computer.
\end{proof}

\begin{rmk}\rm
Computer experiment indicates that when $a=b-27$ is a nonsquare of $\f_q$, we have
\begin{equation}\label{VHb=VFb-1}
|V(H_b)|=|V(F_b)|+1.
\end{equation}
This would imply that $|V(F_b)|>(|V(H_b)|-1)/2$, a key claim in the above proof. 
Although our proof does not need the precise relation \eqref{VHb=VFb-1}, we are curious if it is true.
\end{rmk}

\begin{rmk}\rm
Theorem~\ref{C-tha1} implies that Class III-a has at least $(q-1)/2$ equivalence classes distinguished by the invariant $\theta$. We will see later that Class III-a has precisely $(q-1)/2$ equivalence classes. Therefore, the equivalence classes in Class III-a are completely determined by the invariant $\theta$.
\end{rmk}

\subsection{Class III-b}\

We first determine $\Theta_{\rm b}$, the set of $\theta$ values of functions in Class III-b. 

\begin{thm}\label{thm1}
Assume $p\ge 5$. We have 
\[
\Theta_{\rm b}=27+\mathcal A,
\]
where
\[
\mathcal A=\bigl\{(\mu(\mu^2-9)/(\mu^2-1))^2: \mu\in\f_q\setminus\{0,\pm1,\pm3\}\bigr\}.
\]
\end{thm}

\begin{proof}
We assume $q\ge 71$. For $q\le 67$, the claim is verified by computer. We fix a nonsquare $d$ in $\f_q$.

\medskip
Let $\alpha=a^2\in\Theta_{\rm b}-27$, where $a\in\f_q^*$. Then there exists $(s,t)\in\f_q^2$ with $t(1+s+t)\ne 0$ such that
\begin{equation}\label{27+alpha}
\frac{s^3}{t(1+s+t)}=27+\alpha
\end{equation}
and
\begin{equation}\label{Gsta0a1}
G_{s,t}=X^4-2X^3-sX^2-2tX+t=(X^2+a_1X+a_0)(X^2+b_1X+b_0),
\end{equation}
where $a_0,a_1,b_0,b_1\in\f_q$, and
\begin{equation}\label{du2}
a_1^2-4a_0=du^2
\end{equation}
for some $u\in\f_q^*$. It follows from \eqref{Gsta0a1} that
\begin{equation}\label{s=t=}
\begin{cases}
s=-a_0-b_0-a_1b_1,\cr
t=a_0b_0,
\end{cases}
\end{equation}
and
\[
\begin{cases}
a_1+b_1=-2,\cr
a_1b_0+a_0b_1=-2a_0b_0.
\end{cases}
\]
Thus $b_0$ and $b_1$ can be expressed in terms of $a_0$ and $a_1$:
\begin{equation}\label{b0b1}
\begin{cases}
b_0=\displaystyle\frac{a_0(2+a_1)}{2a_0+a_1},\vspace{0.4em}\cr
b_1=-2-a_1.
\end{cases}
\end{equation}
Note that $2a_0+a_1\ne 0$. (Otherwise, we have $2a_0+a_1=0$ and $a_0(2+a_1)=0$. It follows that $(a_0,a_1)=(0,0)$ or $(1,-2)$. In either case, $X^2+a_1X+a_0$ is not irreducible over $\f_q$.)
By \eqref{27+alpha}, \eqref{s=t=} and \eqref{b0b1}, we have
\begin{align}\label{a^2}
a^2\,&=\alpha=\frac{s^3}{t(1+s+t)}-27\\
&=-\frac{(2+2 a_0+a_1) (4 a_0-a_1^2) (a_0+a_0^2+5 a_0 a_1+2 a_1^2+2 a_0 a_1^2+a_1^3)^2}{a_0^2 a_1 (2+a_1) (1+a_0+a_1)^2 (2 a_0+a_1)}. \nonumber
\end{align}
Since $4a_0-a_1^2=-du^2$, it follows from \eqref{a^2} that 
\begin{equation}\label{cod2}
\frac{d(2+2 a_0+a_1)}{a_1 (2+a_1)(2 a_0+a_1)}=e^2
\end{equation}
for some $e\in\f_q^*$. By choosing a suitable sign for $a$, we have
\begin{equation}\label{cod1}
a=\frac{eu(a_0+a_0^2+5 a_0 a_1+2 a_1^2+2 a_0 a_1^2+a_1^3)}{a_0(1+a_0+a_1)}.
\end{equation}
Making the substitution $a_0=(a_1^2-du^2)/4$ in \eqref{cod2} and \eqref{cod1}, we arrive at 
\[
C_1(a_1)=0,
\]
\[
C_2(a_1)=0,
\]
where
\begin{equation}\label{C1}
C_1(X)=\frac 12 (4 d+2 X d+X^2 d-4 X^2 e^2-4 X^3 e^2-X^4 e^2-d^2 u^2+2 X d e^2 u^2+X^2 d e^2 u^2),
\end{equation}
\begin{align}\label{C2}
C_2(X)=\,&\frac1{16} (-4 a X^2 - 4 a X^3 - a X^4 + 36 X^2 e u + 36 X^3 e u + 
   9 X^4 e u + 4 a d u^2 \\
   &+ 4 a X d u^2 + 2 a X^2 d u^2 - 
   4 d e u^3 - 20 X d e u^3 - 10 X^2 d e u^3 - a d^2 u^4 + 
   d^2 e u^5). \nonumber
\end{align}
We find that
\[
\text{Res}(C_1,C_2;X)=16 d^4 (-4+d u^2)^2 (-a+9 e u+a e^2 u^2-e^3 u^3)^4.
\]
Hence
\[
-a+9 e u+a e^2 u^2-e^3 u^3=0,
\]
i.e.,
\[
a=\frac{eu((eu)^2-9)}{(eu)^2-1}.
\]
Therefore, $\alpha=a^2\in\mathcal A$.

\medskip
On the other hand, assume $\alpha\in\mathcal A$, say $\alpha=a^2$, where
\[
a=\frac{\mu(\mu^2-9)}{\mu^2-1}
\]
for some $\mu\in\f_q\setminus\{0,\pm1,\pm3\}$. Write $\mu=eu$, where $e,u\in\f_q^*$. We find that $C_1(X)$ and $C_2(X)$ defined in \eqref{C1} and \eqref{C2} are related by 
\[
C_2(X)=\frac{-eu^3}{e^2u^2-1}C_1(X).
\]
Using the relation $\mu=eu$, we may express $C_1(X)$ in terms of $\mu$ and $u$:
\[
C_1(X)=\frac {-1}{2u^2}D(X,u),
\]
where 
\begin{align}\label{DXY}
D(X,Y)=\,&4 \mu^2 X^2 + 4 \mu^2 X^3 + \mu^2 X^4 - 4 d Y^2 - 2 d X Y^2\\
& -2 d \mu^2 X Y^2 - d X^2 Y^2 - d \mu^2 X^2 Y^2 + d^2 Y^4.\nonumber
\end{align}

We claim that $D(X,Y)$ is absolutely irreducible. Write $D(X,Y)$ in terms of homogeneous parts:
\[
D=D_2+D_3+D_4,
\]
where
\begin{align*}
&D_2=4 (\mu^2 X^2-d Y^2),\cr
&D_3=2 X (2 \mu^2 X^2-d Y^2-d \mu^2 Y^2),\cr
&D_4=(X^2-d Y^2) (\mu^2 X^2-d Y^2).
\end{align*}
Since $\mu^2\ne 0,1$, we have
\begin{equation}\label{gcd=1}
\gcd(D_2,D_3)=1.
\end{equation}
Assume to the contrary that $D$ is reducible over $\f_q$. Because of \eqref{gcd=1},
\[
D=(A_0+A_1)(B_2+B_3)\quad\text{or}\quad D=(A_1+A_2)(B_1+B_2),
\]
where $A_i,B_j\in\overline\f_q[X,Y]$ are homogeneous of degree $i$ and $j$, respectively. 

\medskip
{\bf Case 1.} Assume $D=(A_0+A_1)(B_2+B_3)$.

We may assume $A_0=1$ and $B_2=D_2$. Since $A_1B_3=D_4$, we have $B_3\mid D_4$. It follows that $\gcd(B_2,B_3)\ne 1$, which contradicts \eqref{gcd=1}.

\medskip
{\bf Case 2.} Assume $D=(A_1+A_2)(B_1+B_2)$.  

Since $A_1B_1=D_2=4 (\mu^2 X^2-d Y^2)$, we may assume that 
\begin{align*}
&A_1=\mu X-\delta Y,\cr
&B_1=4(\mu X+\delta Y),
\end{align*}
where $\delta\in\f_{q^2}\setminus\f_q$ is such that $\delta^2=d$. Note that 
\[
A_2B_2=D_4=(X^2-dY^2)(\mu^2X-dY^2),
\]
and $\gcd(A_1,A_2)=1$, $\gcd(B_1,B_2)=1$ because of \eqref{gcd=1}. Therefore, we have 
\begin{align*}
&A_2=\epsilon(X\pm\delta Y)(\mu X+\delta Y),\cr
&B_2= \epsilon^{-1}(X\mp\delta Y)(\mu X-\delta Y).
\end{align*}
where $\epsilon\in{\overline\f_q}\!^*$. It follows that 
\begin{align*}
D_3\,&=A_1B_2+A_2B_1\cr
&=\epsilon^{-1}(X\mp \delta Y)(\mu X-\delta Y)^2+\epsilon(X\pm\delta Y)(\mu X+\delta Y)^2\cr
&=(\epsilon^{-1}+\epsilon)\mu^2X^3+(\epsilon^{-1}-\epsilon)(\mp\delta^3)Y^3+\text{other terms}.
\end{align*}
Therefore,
\[
\begin{cases}
\epsilon^{-1}+\epsilon=4,\cr
\epsilon^{-1}-\epsilon=0.
\end{cases}
\]
In the above, the second equation gives $\epsilon=\pm 1$. Then $\epsilon^{-1}+\epsilon=2$ or $-2$, which are not equal to 4, which is a contradiction. Hence the claim is proved.

Now, by the Hasse-Weil bound,
\[
|V(D)|\ge q+1-(4-1)(4-2)q^{1/2}=q-6q^{1/2}+1.
\]
(Note that the curve $D(X,Y)=0$ has no rational points at $\infty$.) 

Given $(a_1,u)\in V(D)$, we may recover $a_0,b_0,b_1,s,t$ using  \eqref{du2}, \eqref{b0b1}, \eqref{s=t=}. In doing so, we need to require 
\begin{equation}\label{need}
u\ne0,\quad 2a_0+a_1\ne 0,\quad t(1+s+t)\ne 0.
\end{equation}
Under the condition $2a_0+a_1\ne 0$, we have 
\[
t(1+s+t)=\frac{a_0^2 a_1 (2+a_1) (1+a_0+a_1)^2}{(2 a_0+a_1)^2}.
\]
Note that $a_0\ne0$ and $1+a_0+a_1\ne 0$ since $X^2+a_1X+a_0$ is irreducible over $\f_q$. Therefore, the conditions in \eqref{need} are equivalent to
\[
u\ne0,\quad 2a_0+a_1\ne 0,\quad a_1(2+a_1)\ne 0,
\]
which are equivalent to, in terms of $a_1$ and $u$, 
\[
E(a_1,u)\ne 0,
\]
where 
\[
E(X,Y)=Y(2X+X^2-dY^2) X(2+X).
\]
By B\'ezout's theorem,
\[
|V(D)\cap V(E)|\le 4\cdot 5=20.
\]
Hence 
\[
|V(D)\setminus V(E)|=|V(D)|-|V(D)\cap V(E)|\ge q-6q^{1/2}-19>0.
\]
(Recall that we assumed $q\ge 71$.)

Now choose $(a_1,u)\in V(D)\setminus V(E)$, and let $a_0,b_0,b_1,s,t$ be given by \eqref{du2}, \eqref{b0b1} and \eqref{s=t=}. Then
\[
\theta(f_{s,t})=\frac{s^3}{t(1+s+t)}=27+a^2
\]
and 
\[
G_{s,t}=(X^2+a_1X+a_0)(X^2+b_1X+b_0),
\]
where $X^2+a_1X+a_0$ and $X^2+b_1X+b_0$ are irreducible over $\f_q$. Moreover, $\mu^2\ne 9$ implies that $X^2+a_1X+a_0\ne X^2+b_1X+b_0$. Therefore, $\alpha=a^2\in\Theta_{\rm b}-27$.  
\end{proof}

\begin{cor}\label{cor2}
For $p\ge 5$, $|\Theta_{\rm b}|=\lfloor(q-1)/6\rfloor$.
\end{cor}

\begin{proof}
Consider the map
\[
\begin{array}{cccc}
\phi:&\f_q\setminus\{0,\pm1,\pm3\}&\longrightarrow &\f_q^*\vspace{0.4em}\cr
&x&\xmapsto{\hspace{1em}} & \displaystyle\frac{x(x^2-9)}{x^2-1}.
\end{array}
\]
For $a\in\f_q$, we have
\[
D\bigl(X(X^2-9)-a(X^2-1)\bigr)=4(27+a^2)^2.
\]
If $a\in\text{im}\,\phi$ is such that $27+a^2\ne 0$, then $X(X^2-9)-a(X^2-1)$ has three roots in $\f_q$, so 
$|\phi^{-1}(a)|=3$.
If $a\in\f_q$ is such that $a^2=-27$, which happens if and only if $q\equiv 1\pmod 3$, i.e., if and only if $q\equiv 1\pmod 6$, we have 
\[
X(X^2-9)-a(X^2-1)=\Bigl(X-\frac a3\Bigr)^3.
\]
Since $p\ge 5$, $a^2=-27$ implies $a/3\notin\{0,\pm1,\pm3\}$. Hence in this case, we have $a\in \text{im}\,\phi$ with $\phi^{-1}(a)=\{a/3\}$ and $|\phi^{-1}(a)|=1$. 
Therefore,
\[
|\text{im}\,\phi|=
\begin{cases}
\displaystyle\frac{q-5}3&\text{if}\ q\equiv-1\pmod 6,\vspace{0.4em}\cr
\displaystyle\frac{q-7}3+2=\displaystyle\frac{q-1}3&\text{if}\ q\equiv 1\pmod 6.
\end{cases}
\]
It follows from Theorem~\ref{thm1} that
\[
|\Theta_{\rm b}|=\frac 12|\text{im}\,\phi|=
\begin{cases}
\displaystyle\frac{q-5}6&\text{if}\ q\equiv-1\pmod 6,\vspace{0.4em}\cr
\displaystyle\frac{q-1}6&\text{if}\ q\equiv 1\pmod 6.
\end{cases}
\]
\end{proof}

We will see that Class III-a has precisely $(q-1)/2$ equivalence classes and that Classes III-c and III-d together have precisely 2 equivalence classes. Therefore, we expect Class III-b to have $q-1-(q-1)/2-2=(q-5)/2$ equivalence classes. On the other hand, according to Corollary~\ref{cor2}, the invariant $\theta$ can only distinguish $\lfloor(q-1)/6\rfloor$ equivalence classes in Class III-b. Therefore, the invariant $\theta$ is not strong enough for Class III-b and a stronger invariant is needed.

Let $f_{s,t}$ be in Class III-b. By the proof of Theorem~\ref{thm1},
\[
\theta(s,t)=\frac{s^3}{t(1+s+t)}=27+a^2,
\]
where 
\[
a=\frac{\mu(\mu^2-9)}{\mu^2-1},\quad \mu\in\f_q\setminus\{0,\pm1,\pm3\}.
\]
Moreover,
\begin{equation}\label{eq:Gst}
G_{s,t}=(X^2+a_1X+a_0)(X^2+b_1X+b_0),
\end{equation}
where
\begin{equation}\label{eq:a0}
a_0=\frac 14(a_1^2-du^2),
\end{equation}
$d\in\f_q$ is a fixed nonsquare, and $u\in\f_q$ satisfies
\begin{equation}\label{eq:Da1u}
D(a_1,u)=0,
\end{equation}
where $D(X,Y)$ is given in \eqref{DXY}.
Let $\rho_1,\rho_1^q,\rho_2,\rho_2^q$ be the ramification points of $f_{s,t}$, where $\rho_1,\rho_1^q$ are the roots of $X^2+a_1X+a_0$ and $\rho_2,\rho_2^q$ are the roots of $X^2+b_1X+b_0$. For each $\phi\in\pgl(2,\f_q)$, either $\phi(\{\rho_1,\rho_1^q\})=\{\rho_1,\rho_1^q\}$ and $\phi(\{\rho_2,\rho_2^q\})=\{\rho_2,\rho_2^q\}$, or $\phi(\{\rho_1,\rho_1^q\})=\{\rho_2,\rho_2^q\}$ and $\phi(\{\rho_2,\rho_2^q\})=\{\rho_1,\rho_1^q\}$. In the notation of Section~\ref{sec5}, $\bigl(\phi(\rho_1),\phi(\rho_1^q),\phi(\rho_2),\phi(\rho_2^q)\bigr)$ is a permutation of $(\rho_1,\rho_1^q,\rho_2,\rho_2^q)$ by an element of $\langle V,(1,2)\rangle<S_4$. It follows from \eqref{C1234} that 
\[
C\bigl(\phi(\rho_1),\phi(\rho_1^q),\phi(\rho_2),\phi(\rho_2^q)\bigr)=C((\rho_1,\rho_1^q,\rho_2,\rho_2^q)\kern0.6em \text{or}\kern0.6em C((\rho_1,\rho_1^q,\rho_2,\rho_2^q)^{-1},
\]
where $C(\ )$ denotes the cross ratio. Define
\begin{equation}\label{eq:lambda}
\lambda=\lambda(f_{s,t})=C((\rho_1,\rho_1^q,\rho_2,\rho_2^q)+C((\rho_1,\rho_1^q,\rho_2,\rho_2^q)^{-1}.
\end{equation}
Then $\lambda$ is an invariant under equivalence. Note that $\lambda$ is a rational function in $\rho_1,\rho_1^q,\rho_2,\rho_2^q$ which is symmetric in $\rho_1,\rho_1^q$ and $\rho_2,\rho_2^q$. Therefore, $\lambda$ can be expressed as a rational function in $a_0,a_1,b_0,b_1$:
\[
\lambda=\frac{-2 a_0^2+(-8 a_0+2 (-2 a_0+a_1^2)) b_0-2 b_0^2+2 a_0 a_1 b_1+2 a_1 b_0 b_1+(2 a_0-a_1^2) b_1^2}{-a_0^2+(2 a_0-a_1^2) b_0-b_0^2+a_0 a_1 b_1+a_1 b_0 b_1-a_0 b_1^2}.
\]
Furthermore, with the substitution \eqref{b0b1}, $\lambda$ can be expressed in terms of $a_0$ and $a_1$ only:
\begin{dmath}
\lambda=(8 a_0^2+16 a_0^3+8 a_0^4+16 a_0^2 a_1+16 a_0^3 a_1+24 a_0^2 a_1^2+16 a_0 a_1^3+16 a_0^2 a_1^3+4 a_1^4+16 a_0 a_1^4+4 a_0^2 a_1^4+4 a_1^5+4 a_0 a_1^5+a_1^6) /\bigl(4 a_0 (1+a_0+a_1) (a_0+a_0^2+5 a_0 a_1+2 a_1^2+2 a_0 a_1^2+a_1^3)\bigr).
\end{dmath}
Making the substitution $a_0=(a_1^2-du^2)/4$ in the above and use \eqref{eq:Da1u} to reduce the result, we get
\begin{equation}\label{lba-m}
\lambda=\frac{2(9-2\mu^2+\mu^4)}{(\mu^2-1)(\mu^2-9)}.
\end{equation}
Note that $\lambda$ depends only on $\mu^2$, but not on $a_1$ and $u$. The invariant $\theta$ also depends on $\mu^2$ only:
\begin{equation}\label{eq:the-mu2}
\theta=27+\frac{\mu^2(\mu^2-9)^2}{(\mu^2-1)^2}.
\end{equation}
The pair $(\theta,\lambda)$ is a stronger invariant than $\theta$ alone. We now determine the number of possible values of $(\theta,\lambda)$.

Define
\[
\begin{array}{cccc}
\Psi:&\f_q\setminus\{0,\pm1,\pm3\}&\longrightarrow &\f_q^2\vspace{0.2em}\cr
&x&\xmapsto{\hspace{1em}} &(\Psi_1(x),\Psi_2(x)),
\end{array}
\]
where 
\[
\Psi_1(x)=27+\frac{x^2(x^2-9)^2}{(x^2-1)^2},
\]
\[
\Psi_2(x)=\frac{2(9-2x^2+x^4)}{(x^2-1)(x^2-9)}.
\]
For $x,x'\in\f_q\setminus\{0,\pm1,\pm3\}$, we have
\[
\Psi_1(x')=\Psi_1(x)\ \Leftrightarrow\ x'=\pm x,\ \pm\frac{x+3}{x-1},\ \pm\frac{x-3}{x+1}.
\]
Moreover,
\[
\Psi_2(\pm x)-\Psi_2\Bigl(\pm\frac{x+3}{x-1}\Bigr)=\frac{(3+x^2)^2(-3+6x+x^2)}{4x(x^2-1)(x^2-9)},
\]
\[
\Psi_2(\pm x)-\Psi_2\Bigl(\pm\frac{x-3}{x+1}\Bigr)=-\frac{(3+x^2)^2(-3-6x+x^2)}{4x(x^2-1)(x^2-9)},
\]
\[
\Psi_2\Bigl(\pm\frac{x+3}{x-1}\Bigr)-\Psi_2\Bigl(\pm\frac{x-3}{x+1}\Bigr)=-\frac{(-3+x^2)(3+x^2)^2}{2x(x^2-1)(x^2-9)}.
\]
It follows that
\[
\Bigl(\Psi_2(\pm x),\, \Psi_2\Bigl(\pm\frac{x+3}{x-1}\Bigr),\, \Psi_2\Bigl(\pm\frac{x-3}{x+1}\Bigr)\Bigr)=
\begin{cases}
(1,1,1)&\text{if}\ x=\pm\sqrt{-3},\vspace{0.2em}\cr
(5/2,5/2,-2)&\text{if}\ x=-3\pm 2\sqrt 3,\vspace{0.2em}\cr
(5/2,-2, 5/2)&\text{if}\ x=3\pm 2\sqrt 3,\vspace{0.2em}\cr
(-2,5/2,5/2)&\text{if}\ x=\pm\sqrt 3.
\end{cases}
\]
For the other values of $x$, $\Psi_2(\pm x)$, $\Psi_2\bigl(\pm(x+3)/(x-1)\bigr)$, $\Psi_2\bigl(\pm(x-3)/(x+1)\bigr)$ are distinct. Note that $\sqrt{-3}\in\f_q$ if and only if $q\equiv 1\pmod 6$ and that $\sqrt 3\in\f_q$ if and only if $q\equiv \pm1\pmod{12}$.
Moreover,
\[
\Psi_1(x)=\begin{cases}
0&\text{if}\ x=\pm\sqrt{-3},\vspace{0.4em}\cr
54&\text{if}\ x=-3\pm2\sqrt3,\; 3\pm2\sqrt3,\;\pm\sqrt3,
\end{cases}
\]
and $\Psi_1(x)\in\Theta_{\rm b}\setminus\{0,54\}$ for other values of $x$.

Since $\Psi(\mu)=(\theta,\lambda)$, the above discussion allows us to enumerate the possible values of $(\theta,\lambda)$ in Table~\ref{tb:the-lam}.

\begin{table}[ht]
\caption{Values of $\mu$, $\theta$, $\lambda$}\label{tb:the-lam}
   \renewcommand*{\arraystretch}{1.4}
    \centering
     \begin{tabular}{c|c|c|c}
         \hline
         $\mu$ & $\theta$ & $\lambda$  & condition \\ \hline
         $\pm\sqrt{-3}$ & 0 & 1& $q\equiv 1\pmod 6$ \\ \hline
         $\begin{array}{c} \pm \sqrt3 \cr \pm(3+2\sqrt3),\, \pm(3-2\sqrt3)\end{array}$ & 54 & $\begin{array}{c} -2 \cr 5/2\end{array}$ & $q\equiv \pm1\pmod{12}$\\ \hline
         $\begin{array}{l} \ne0,\,\pm1,\,\pm3,\, \pm\sqrt{-3},\,\pm\sqrt 3\cr \pm(3+2\sqrt3),\, \pm(3-2\sqrt3)\end{array}$ & $\Theta_{\rm b}\setminus\{0,54\}$  &  3 values for each $\theta$  \\ \hline
     \end{tabular}
\end{table}

From Table~\ref{tb:the-lam}, and using the fact that $|\Theta_{\rm b}|=\lfloor(q-1)/6\rfloor$, we see that the number of pairs $(\theta,\lambda)$ is 
\[
\begin{cases}
\displaystyle\frac{q-5}2-1&\text{if}\ q\equiv\pm1\pmod{12},\vspace{0.4em}\cr
\displaystyle\frac{q-5}2&\text{if}\ q\not\equiv\pm1\pmod{12}.
\end{cases}
\]
This presents a subtle situation: the invariant $(\theta,\lambda)$ determines the all equivalence classes in Class III-b for $q\not\equiv\pm1\pmod{12}$, but fails to distinguish two equivalence classes for $q\equiv\pm1\pmod{12}$. Computer experiment indicates that among the functions in Class III-b with $(\theta,\lambda)=(54,\,5/2)$, the ones with $\mu^2=(3+2\sqrt3)^2$ and the ones with $\mu^2=(3-2\sqrt3)^2$ belong to two different equivalence classes. Therefore, we are convinced that $\mu^2$ is indeed an invariant.

We first express $\mu^2$ in terms of $a_0$ and $a_1$ in \eqref{eq:Gst}. Recall that in \eqref{cod2}, $e=\mu/u$ and $du^2=a_1^-4a_0$. Hence by \eqref{cod2},
\begin{equation}\label{mu2=}
\mu^2=\frac{du^2(2+2a_0+a_1)}{a_1(2+a_1)(2a_0+a_1)}=\frac{(a_1^2-4a_0)(2+2a_0+a_1)}{a_1(2+a_1)(2a_0+a_1)}.
\end{equation}
Set
\begin{equation}\label{eq:M}
M(X,Y)=\frac{(Y^2-4X)(2+2X+Y)}{Y(2+Y)(2X+Y)} 
\end{equation} 
so that $\mu^2=M(a_0,a_1)$. It is easy to verify that $M(a_0,a_1)=M(b_0,b_1)$, where $X^2+b_1X+b_0$ is the other irreducible factor of $G_{s,t}$ and $b_0,b_1$ are given in \eqref{b0b1}. Therefore, $\mu^2$ depends only on $(s,t)$ but not on the choice of the irreducible factor of $G_{s,t}$.

\begin{thm}\label{thm:mu2}
$\mu^2$ is an invariant for Class III-b, that is, if $f_{s,t}\sim f_{s',t'}$ are in Class III-b and
\[
G_{s,t}=(X^2+a_1X+a_0)(X^2+b_1X+b_0),
\]
\[
G_{s',t'}=(X^2+a_1'X+a_0')(X^2+b_1'X+b_0'),
\]
then
\[
M(a_0,a_1)=M(a_0',a_1').
\]
\end{thm}

\begin{proof}
We have $f_{s',t'}=\psi\circ f_{s,t}\circ\phi$ for some $\phi,\psi\in\pgl(2,\f_q)$. Write $\phi=\left[\begin{smallmatrix}a&b\cr c&d\end{smallmatrix}\right]$, so $\phi^{-1}=\left[\begin{smallmatrix}d&-b\cr -c&a\end{smallmatrix}\right]$. Let $\alpha,\beta$ be the roots of $X^2+a_1X+a_0$ ($\beta=\alpha^q$). Then $\phi^{-1}(\alpha)$ is a ramification point of $f_{s',t'}$. We may assume $\phi^{-1}(\alpha)$ is a root of $X^2+a_1'X+a_0'$. The other root of $X^2+a_1'X+a_0'$ is $(\phi^{-1}(\alpha))^q=\phi^{-1}(\alpha^q)=\phi^{-1}(\beta)$. Hence
\[
a_0'=\phi^{-1}(\alpha)\phi^{-1}(\beta),\quad a_1'=-(\phi^{-1}(\alpha)+\phi^{-1}(\beta)).
\]
Both $a_0'$ and $a_1'$ are rational functions in $\alpha$ and $\beta$ which are symmetric in $\alpha$ and $\beta$. Therefore, both $a_0'$ and $a_1'$ are rational functions in $a_0$ and $a_1$: 
\[
a_0'=\frac{b^2+d^2 a_0+b d a_1}{a^2+c^2 a_0+a c a_1},
\]
\begin{equation}\label{a11}
a_1'=\frac{2 a b+2 c d a_0+(b c+a d) a_1}{a^2+c^2 a_0+a c a_1},
\end{equation}
where the denominators are nonzero. It follows that 
\begin{align*}
M(a_0',a_1')=\,&
-\frac 1D (4 a_0-a_1^2) (b c-a d)^2\cr
&\cdot (2 a^2+2 a b+2 b^2+a a_1 (2 c+d)+a_1 b (c+2 d)+2 a_0 (c^2+c d+d^2)),
\end{align*}
where
\begin{align*}
D=\,&(2 a b+a_1 b c+a a_1 d+2 a_0 c d) (2 a^2+2 a b+a_1 b c+2 a_0 c (c+d)+a a_1 (2 c+d))\cr
&\cdot (2 b^2+2 a_0 d (c+d)+a_1 b (c+2 d)+a (2 b+a_1 d)).
\end{align*}
Since $M(a_0',a_1')$ is defined, we have $a_1'(2+a_1')(2a_0'+a_1')\ne0$, which implies $D\ne 0$. Now we compute
\[
M(a_0,a_1)-M(a_0',a_1')=\frac{8 (-4 a_0+a_1^2)N}{a_1 (2 + a_1) (2 a_0 + a_1)D},
\]
where
\begin{equation}\label{N}
N=2 a^4 b^2 + 2 a^4 a_0 b^2 + a^4 a_1 b^2 +\cdots +2 a_0^4 c^2 d^4+a_0^3 a_1 c^2 d^4.
\end{equation}
(The Mathematica code for the above computation is given in Appendix~\ref{app-b}.) It suffices to show that $N=0$.

\medskip
{\bf Case 1.} Assume $\phi\in S_3$. Note that $S_3$ is generated by $\phi_1=\left[\begin{smallmatrix}1&-1\cr 0&-1\end{smallmatrix}\right]$ and $\phi_2=\left[\begin{smallmatrix}0&1\cr 1&0\end{smallmatrix}\right]$. We find that $N=0$ when $\phi=\phi_1, \phi_2$. Or, one can check directly that for $\phi=\phi_1,\phi_2$, $M(a_0',a_1')=M(a_0,a_1)$.

\medskip
{\bf Case 2.} Assume $\phi\notin S_3$.  We may assume $a_1\ne -1$. (If $a_1=-1$, we must have $a_0\ne 1$, since otherwise, by \eqref{b0b1}, $(b_0,b_1)=(1,-1)=(a_0,a_1)$, which is a contradiction. Setting $\phi=\phi_2=\left[\begin{smallmatrix}0&1\cr 1&0\end{smallmatrix}\right]$ in \eqref{a11}, we get $a_1'=a_1/a_0\ne-1$. Simply replace $f_{s,t}$ with $\psi\circ f_{s,t}\circ\phi_1$, where $\psi\in\pgl(\f_q,2)$ is such that $\psi\circ f_{s,t}\circ\phi_1$ is of the form $f_{s_1,t_1}$ for some $(s_1,t_1)$.) By \eqref{notinS3}, $acd(a-c)(c+d)\ne 0$ and \eqref{st} holds. By \eqref{s=t=} and \eqref{b0b1}, we can express $s$ and $t$ in terms of $a_0, a_1$:
\begin{equation}\label{eq:sta0a1}
\begin{cases}
s=\displaystyle-\frac{2 a_0+2 a_0^2-2 a_0 a_1-2 a_1^2-2 a_0 a_1^2-a_1^3}{2 a_0+a_1},\vspace{0.4em}\cr
t=\displaystyle\frac{a_0^2 (2 + a_1)}{2 a_0 + a_1}.
\end{cases}
\end{equation}
Combining \eqref{st} with \eqref{eq:sta0a1} gives
\begin{align}\label{s-abcd}
&\frac{a^2 d+2 a b c+2 a b d-2 a c d-a d^2+b^2 c-b c^2-2 b c d}{c d (c+d)}\\
=\,&-\frac{2 a_0+2 a_0^2-2 a_0 a_1-2 a_1^2-2 a_0 a_1^2-a_1^3}{2 a_0+a_1},\nonumber
\end{align}
\begin{equation}\label{t-abcd}
-\frac{a b (a+b)}{c d (c+d)}=\frac{a_0^2 (2 + a_1)}{2 a_0 + a_1}.
\end{equation}
We may assume $a=1$ (since $a\ne 0$). Write \eqref{s-abcd} and \eqref{t-abcd} as quadratic equations in $a_0$:
\begin{equation}\label{E1}
A_2a_0^2+A_1a_0+A_0=0,
\end{equation}
\begin{equation}\label{E2}
B_2a_0^2+B_1a_0+B_0=0,
\end{equation}
where
\[
\begin{cases}
A_0=-a_1 (-2 b c-b^2 c+b c^2-d-2 b d+2 c d+2 b c d\cr
\kern2.5em +2 a_1 c^2 d+a_1^2 c^2 d+d^2+2 a_1 c d^2+a_1^2 c d^2),\vspace{0.4em}\cr
A_1=-2 (-2 b c-b^2 c+b c^2-d-2 b d+2 c d+2 b c d-c^2 d\cr
\kern2.5em+a_1 c^2 d+a_1^2 c^2 d+d^2-c d^2+a_1 c d^2+a_1^2 c d^2),\vspace{0.4em}\cr
A_2=2 c d (c+d),
\end{cases}
\]
\[
\begin{cases}
B_0=a_1 b (1 + b),\cr
B_1= 2 b (1 + b),\cr
B_2= (2 + a_1) c d (c + d).
\end{cases}
\]
Computing $\eqref{E1}\times B_2-\eqref{E2}\times A_2$ gives
\begin{equation}\label{eqa0}
-2cd(c+d)C_1a_0=a_1cd(c+d)C_0=0,
\end{equation}
where 
\begin{align*}
C_0=\,&
2 b+2 b^2-4 b c-2 a_1 b c-2 b^2 c-a_1 b^2 c+2 b c^2+a_1 b c^2-2 d-a_1 d-4 b d\cr
&-2 a_1 b d+4 c d+2 a_1 c d+4 b c d+2 a_1 b c d+4 a_1 c^2 d+4 a_1^2 c^2 d\cr
&+a_1^3 c^2 d+2 d^2+a_1 d^2+4 a_1 c d^2+4 a_1^2 c d^2+a_1^3 c d^2,
\end{align*}
\begin{align*}  
C_1=\,&  
2 b+2 b^2-4 b c-2 a_1 b c-2 b^2 c-a_1 b^2 c+2 b c^2+a_1 b c^2-2 d-a_1 d-4 b d\cr
&-2 a_1 b d+4 c d+2 a_1 c d+4 b c d+2 a_1 b c d-2 c^2 d+a_1 c^2 d+3 a_1^2 c^2 d\cr
&+a_1^3 c^2 d+2 d^2+a_1 d^2-2 c d^2+a_1 c d^2+3 a_1^2 c d^2+a_1^3 c d^2.
\end{align*}
If $C_1=0$, then $C_0=0$. We then have
\[
0=C_0-C_1=(1+a_1) (2+a_1) c d (c+d)\ne 0,
\]
which is a contradiction. Therefore, $C_1\ne 0$. Now by \eqref{eqa0},
\begin{equation}\label{a0=}
a_0=-\frac{a_1C_0}{2C_1}.
\end{equation}
Making this substitution in \eqref{E2} gives 
\[
\frac{a_1 (2 + a_1)^2 c d (c + d)K}{4C_1^2}=0,
\]
where
\[
K=-4 b^2-8 b^3-4 b^4+8 b^2 c+\cdots +6 a_1^5 c^2 d^4+a_1^6 c^2 d^4.
\]
Hence $K=0$. Finally, making the substitution \eqref{a0=} in \eqref{N} gives
\[
N=\frac{K\cdot(-64b^4-256b^5-\cdots-a_1^9c^4d^8)}{8C_1^4}=0.
\]
This completes the proof. (The Mathematica code for the above computation is given in Appendix~\ref{app-b}.)
\end{proof}

\begin{rmk}\rm
Since $\mu$ takes values in $\f_q\setminus\{0,\pm1,\pm3\}$, $\mu^2$ takes precisely $(q-5)/2$ values (the squares in $\f_q$ other than 0,1,9). Hence $\mu^2$ distinguishes at least $(q-5)/2$ equivalence classes in Class III-b. We will see that Class III-b has precisely $(q-5)/2$ equivalence classes. Therefore, $\mu^2$ determines the equivalence classes in Class III-b completely. The difference between the invariants $\mu^2$ and $(\theta,\lambda)$ is subtle: the latter determines almost all equivalence classes in Class III-b except for $(\theta,\lambda)=(54,5/2)$ where it fails to distinguish two equivalence classes.
\end{rmk}
 
\subsection{Class III-c}\

Recall that Class III-c consists of $f_{s,t}$ such that $G_{s,t}$ is a square of an irreducible quadratic over $\f_q$. and $\Theta_{\rm c}=\{\theta(s,t): f_{s,t}\ \text{is in Class III-c}\}$. By Corollary~\ref{C-tha}, we know that $\Theta_{\rm c}\subset\{27\}$.

\begin{lem}\label{tha=27}
Assume $p\ge 5$. Then
$(s,t)\in\f_q^2$ satisfies 
\[
\theta(s,t)=\frac{s^3}{t(1+s+t)}=27
\]
if and only if
\begin{itemize}
\item[(i)]
\[
\begin{cases}
s=\displaystyle\frac 34(-1+u^2),\vspace{0.4em}\cr
t=-\displaystyle\frac 18(1+u)^3,
\end{cases}
\]
for some $u\in\f_q\setminus\{\pm 1\}$, or
\medskip
\item[(ii)] $(s,t)=(-3,1)$.
\end{itemize}
\end{lem}

\begin{proof}
($\Rightarrow$) We have
\begin{equation}\label{27t2}
27t^2+27(1+s)t-s^3=0.
\end{equation}
Hence
\begin{equation}\label{3+4s}
(27(1+s))^2+4\cdot 27s^3=27(3+s)^2(3+4s)
\end{equation}
is a square in $\f_q$. Thus, either $3+s=0$ or $3+4s=3u^2$ for some $u\in\f_q$. 

First assume $3+s=0$. It follows from \eqref{27t2} that $t=1$, which is (ii).
Now assume $3+4s=3u^2$, i.e.,
\[
s=\frac 34(-1+u^2).
\]
Then by \eqref{27t2} and \eqref{3+4s}, 
\[
t=\frac 1{2\cdot 27}\bigl[-27(1+s)+9u(3+s)\bigr]=-\frac 18(1+u)^3.
\]
Since 
\[
0\ne t(1+s+t)=\frac 1{64}(-1+u)^3(1+u)^3,
\]
we have $u\ne\pm 1$. Hence we have (i).

\medskip
($\Leftarrow$) Direct computation.
\end{proof}

\begin{thm}\label{C-tha3}
For $p\ge 5$, 
\[
\Theta_{\rm c}=
\begin{cases}
\emptyset& \text{\rm if}\ q\equiv 1 \pmod 6,\cr
\{27\}& \text{\rm if}\ q\not\equiv 1 \pmod 6.
\end{cases}
\]
\end{thm}

\begin{proof}
Let $(s,t)\in\f_q^2$ be such that $\theta(s,t)=27$.
In Lemma~\ref{tha=27} (i), 
\[
G_{s,t}=\frac 18(1+u-2X)^2(-1-u+2(u-1)X+2X^2),
\]
which is not a square of an irreducible quadratic over $\f_q$. Hence $f_{s,t}$ is not in Class III-c. 
In Lemma~\ref{tha=27} (ii), we have $G_{-3,1}=(1-X+X^2)^2$, where $1-X+X^2$ is irreducible over $\f_q$ if and only if $q\not\equiv 1\pmod 3$, i.e., $q\not\equiv 1\pmod 6$ since $q$ is odd. Thus $f_{-3,1}$ is in Class III-c if and only if $q\not\equiv 1\pmod 6$.
Therefore, $27\in\Theta_3$ if and only if $q\not\equiv 1\pmod 6$.
\end{proof}

\begin{cor}\label{cor:III-d}
Assume $p\ge 5$. If $q\equiv 1\pmod 6$, Class III-c is empty. If $q\not\equiv 1\pmod 6$, Class III-c has exactly one equivalence class represented by $f_{-3,1}$.
\end{cor}

\subsection{Class III-d}\label{sec6.7}\

In Class III-d, $G_{s,t}$ has at least one root $u\in\f_q$ and all roots $x$ of $G_{s,t}$ in $\f_q$ satisfy $x^3=-t$.
Thus we have 
\begin{equation}\label{ts}
\begin{cases}
t=-u^3,\cr
s=3(-u+u^2),
\end{cases}
\end{equation}
where $0\ne t(1+s+t)=(-1+u)^3u$, i.e., $u\ne 0,1$. Moreover,
\[
G_{s,t}=(X-u)^2(-u+2(u-1)X+X^2),
\]
and $\theta(s,t)=27$.
The function $f_{s,t}$ is in Class III-d if and only if
\begin{itemize}
\item[(i)] 
$-u+2(u-1)X+X^2$ is irreducible over $\f_q$, or
\item[(ii)]
$-u+2(u-1)X+X^2$ is reducible over $\f_q$ and divides $X^3+t=X^3-u^3$.
\end{itemize}

Condition (i) is satisfied if and only if $(2(u-1))^2+4u=4(1-u+u^2)$ is a nonsquare in $\f_q$, i.e., $1-u+u^2$ is a nonsquare in $\f_q$. Since $p>3$, there always exits $u\in\f_q$ satisfying this condition. In fact, let $d$ be a nonsquare of $\f_q$. Then $|V(1-X+X^2-dY^2)|$, the number of zeros of $1-X+X^2-dY^2\in\f_q[X,Y]$ in $\f_q^2$, is known. Write
\[
1-X+X^2-dY^2=\Bigl(X-\frac 12\Big)^2-dY^2+\frac 34.
\]
By \cite[Lemma~6.55]{Hou-ams-gsm-2018}, $|V(1-X+X^2-dY^2)|=q\pm q^{1/2}$. Hence
\begin{align*}
&|\{(x,y)\in\f_q^2:1-x+x^2-dy^2=0,\ y\ne 0\}|\cr
&\ge|V(1-X+X^2-dY^2)|-2\ge q-q^{1/2}-2>0
\end{align*}
for $q\ge 5$.

In (ii),
\[
0=\text{Res}(X^3-u^3,\,-u+2(u-1)X+X^2;\,X)=9(-1+u)u^3(1-u+u^2),
\]
whence $1-u+u^2=0$. Such $u\in\f_q$ exists if and only if $q\equiv 1\pmod 3$. Since $1-u+u^2=0$, we have $(s,t)=(-3,1)$ (by \eqref{ts}) and
\[
G_{s,t}=(1-X+X^2)^2=\bigl((X-u)(X+u^2)\bigr)^2,
\]
which divides $(X^3-u^3)^2=(X^3+1)^2$.

\medskip

To summarize, we have the following theorem.

\begin{thm}\label{thm:III-d}
Assume $p\ge 5$. Let $u\in\f_q$ be such that $1-u+u^2$ is a nonsquare in $\f_q$ and let $(s,t)$ be given by \eqref{ts}. Then $f_{s,t}$ is in Class III-d. In this case, $G_{s,t}=(X-u)^2(-u+2(u-1)X+X^2)$, where 
$-u+2(u-1)X+X^2$ is irreducible over $\f_q$, and the ramification type of $f_{s,t}$ is $(2/2,\,2/2,\,3/1)$. If $q\equiv 1\pmod 6$, $f_{-3,1}$ also belongs to Class III-d. In this case, $G_{s,t}=(1-X+X^2)^2=((x+\epsilon)(x+\epsilon^2))^2$, where $\epsilon\in\f_q$, $o(\epsilon)=3$, and the ramification type of $f_{-3,1}$ is $(3/1,\,3/1)$.
\end{thm}

We will see that in Theorem~\ref{thm:III-d}, all $f_{s,t}$ with ramification type $(2/2,\,2/2,\,3/1)$ are equivalent.

\subsection{Summary of Class III}\label{sec6.8}\ 

Summarizing the results of Sections~\ref{sec6.4} -- \ref{sec6.7}, we have the numbers of equivalence classes in Classes III-a through III-d tabulated in Table~\ref{tb:III}. 

\begin{table}[ht]
\caption{Number of equivalence classes in Class III}\label{tb:III}
   \renewcommand*{\arraystretch}{1.4}
    \centering
     \begin{tabular}{c|c|c|c|c|c}
         \hline
           &  III-a & III-b & III-c & III-d & III \\ \hline
         $\begin{array}{c}q\equiv 1\pmod{12}\cr q\equiv 5\pmod{12}\cr q\equiv 7\pmod{12}\cr q\equiv -1\pmod{12}\end{array}$  & $\displaystyle \ge\frac{q-1}2$ & $\displaystyle \ge\frac{q-5}2$ & $\begin{array}{c} 0\cr 1\cr 0\cr 1 \end{array}$ & $\begin{array}{c} \ge2\cr \ge1\cr \ge2\cr \ge1 \end{array}$ & $\ge q-1$ \\ \hline
     \end{tabular}
\end{table}
Since the total number of equivalence classes in Class III is known to be $q-1$, all the ``$\le$'' in Table~\ref{tb:III} are ``$=$''. Therefore, we have a complete classification of Class III. 

Let $\f_q^{(2)}=\{x^2:x\in\f_q\}$. For each $a\in\f_q\setminus\f_q^{(2)}$, find $(s,t)\in\Omega$ such that $G_{s,t}$ is irreducible over $\f_q$ and $\theta(s,t)=27+a$; let $\alpha_a=f_{s,t}$. For each $b\in\f_q^{(2)}\setminus\{0,1,9\}$, find $(s,t)\in\Omega$ such that $G_{s,t}=(X^2+a_1X+a_0)(X^2+b_1X+b_0)$, where $X^2+a_1X+a_0$ and $X^2+b_1X+b_0$ are irreducible over $\f_q$ and $M(a_0,a_1)=b$; let $\beta_b=f_{s,t}$. Choose $u\in\f_q$ such that $1-u+u^2\notin\f_q^{(2)}$ and let $\gamma_u=f_{3u(u-1),-u^3}$. Then the representatives of the equivalence classes in Class III are given in Table~\ref{tb:III-cls}.

\begin{table}[ht]
\caption{Classification of Class III, $p\ge 5$}\label{tb:III-cls}
   \renewcommand*{\arraystretch}{1.4}
    \centering
     \begin{tabular}{c|c}
         \hline
         class & representatives of equivalence classes \\ \hline
         III-a & $\alpha_a:\ a\in\f_q\setminus\f_q^{(2)}$ \\
         III-b & $\beta_b:\ b\in\f_q^{(2)}\setminus\{0,1,9\}$ \\
         III-c and III-d & $\gamma_u,\ f_{-3,1}$
            \\ \hline
     \end{tabular}
\end{table}

\section{Class III with $p=3$}\label{sec7}

We assume $p=3$ throughout this section. To obtain a classification of Class III in this case, we only have to modify some results and arguments in Section~\ref{sec6}.

First, Section~\ref{sec6.4} on Class III-a remains valid for $p=3$. 

For Class III-b, we need to slightly modify the proof of Theorem~\ref{thm1}. In that proof, we assume to the contrary that $D(X,Y)$ is not absolutely irreducible and arrive at $\epsilon=\pm1$. When $\epsilon=-1$, $\epsilon^{-1}+\epsilon=-2=4$, which is not a a contradiction for $p=3$. However, we then have
\begin{align*}
&2 X (2 \mu^2 X^2-d Y^2-d \mu^2 Y^2)=D_3=A_1B_2+A_2B_1\cr
=\,&-(X\mp \delta Y)(\mu X-\delta Y)^2-(X\pm\delta Y)(\mu X+\delta Y)^2\cr
=\,&-X\bigl[(\mu X-\delta Y)^2+(\mu X+\delta Y)^2\bigr]\pm\delta Y\bigl[(\mu X-\delta Y)^2-(\mu X+\delta Y)^2\bigr]\cr
=\,&-2X(\mu^2 X^2+\delta^2 Y^2)\pm\delta Y(-2\mu\delta XY)\cr
=\,&2 X (2 \mu^2 X^2-d Y^2\mp d \mu Y^2).
\end{align*}
It follows that $\mu^2=\pm\mu$, whence $\mu^2=0$ or 1, which is a contradiction. So Theorem~\ref{thm1} is still valid. Since $3=0$, the invariant $\mu^2$ takes values in $\f_q^{(2)}\setminus\{0,1\}$, so it takes $(q-3)/2$ different values. Corollary~\ref{cor2} still holds and $|\Theta_{\rm b}|=(q-3)/6$. For each $\theta\in\Theta_{\rm b}$, there are precisely 3 corresponding values of $\lambda$ (see Table~\ref{tb:the-lam}). Hence $(\theta,\lambda)$ also takes $(q-3)/2$ different values. Class III-b has precisely $(q-3)/2$ equivalence classes, so either $\mu^2$ or $(\theta,\lambda)$ determines the equivalence classes completely.

Class III-c is empty. If this class contains a function $f$, then its ramification type would be $(2/2,\,2/2)$ or $(3/2,\,3/2)$, which gives $\deg(\text{Diff}(\overline\f_q(X)/\overline\f_q(f))=2$ or $>4$. However, by the Hurwitz genus formula, $\deg(\text{Diff}(\overline\f_q(X)/\overline\f_q(f))=4$.

For Class III-d, Section~\ref{sec6.7} is modified as follows. Case (i) does not occur since $1-u+u^2=(u+1)^2$ is a square in $\f_q$. Case (ii) is still valid. Since $p=3$, we have $-u+2(u-1)X+X^2=(X-u)(X+1)$, which divides $X^3-u^3$ if and only if $u=-1$. In this case, $(s,t)=(0,1)$.

The classification of Class III is summarized in Table~\ref{tb:IIIp=3}, where $\alpha_a, \beta_b$ are defined in Section~\ref{sec6.8}.

\begin{table}[ht]
\caption{Classification of Class III, $p=3$}\label{tb:IIIp=3}
   \renewcommand*{\arraystretch}{1.4}
    \centering
     \begin{tabular}{c|c}
         \hline
         class & representatives of equivalence classes \\ \hline
         III-a & $\alpha_a:\ a\in\f_q\setminus\f_q^{(2)}$ \\
         III-b & $\beta_b:\ b\in\f_q^{(2)}\setminus\{0,1\}$ \\
         III-d & $f_{0,1}$
            \\ \hline
     \end{tabular}
\end{table}

\section{Direct Proofs}

We have seen that in Class III-a, the invariant $\theta$ determines the equivalence classes, and in Class III-b, the invariant $(\theta,\lambda)$ determines the equivalence classes with $(\theta,\lambda)\ne(54,5/2)$, that is, if $f_{s,t}, f_{s',t'}$ in Class III-a are such that $\theta(f_{s,t})=\theta(f_{s',t'})$, then $f_{s,t}\sim f_{s',t'}$, and if $f_{s,t}, f_{s',t'}$ in Class III-b are such that $(\theta(f_{s,t}),\lambda(f_{s,t}))=(\theta(f_{s',t'}),\lambda(f_{s',t'}))$, then $f_{s,t}\sim f_{s',t'}$. These conclusions rely the prior knowledge of the total number of equivalence classes of degree 3 rational functions, which was obtained by an entirely different method \cite{Hou-AA-2025}. Naturally, one asks if it is possible to prove the above conclusions directly without using the total number of equivalence classes of degree 3 rational functions.In this section, we provide such direct proofs using cross ratios. However, these proofs fail to cover a few exceptional cases, underpinning the importance of knowing the total number of equivalence classes in advance. The proofs in this section are not needed for our classification, but they allow us to look at the question from a different perspective. We will follow the notation of Section~\ref{sec5}.

\begin{thm}\label{pqthm}
Assume $p=\ch\f_q\ge 5$. If $f,g$ are in Case III-a and $\theta(f)=\theta(g)\ne 0$, then $f\sim g$ over $\f_q$.
\end{thm}

\begin{proof}
We know that $f\sim g$ over $\overline\f_q$, i.e., $f=\alpha\circ g\circ\beta$ for some $\alpha,\beta\in\pgl(2,\overline\f_q)$. (In fact, $f\sim g$ over $\f_{q^4}$.) Let $\sigma(x)=x^q$ for $x\in\overline\f_q$. Let $\rho,\sigma(\rho),\sigma^2(\rho),\sigma^3(\rho)$ be the ramification points of $f$ and $\omega,\sigma(\omega),\sigma^2(\omega),\sigma^3(\omega)$ be the ramification points of $g$. Then
\begin{equation}\label{beta}
\beta(\sigma^j(\rho))=\sigma^{i_j}(\omega),\quad 0\le j\le 3,
\end{equation}
where $(i_0,\dots,i_3)$ is a permutation of $(0,1,2,3)$. Let
\[
l=C(\rho,\sigma(\rho),\sigma^2(\rho),\sigma^3(\rho)).
\]
Then
\[
C(\sigma^{i_0}(\omega),\dots,\sigma^{i_3}(\omega))=C(\rho,\sigma(\rho),\sigma^2(\rho),\sigma^3(\rho))=l.
\]
Applying $\sigma$ to the above gives
\begin{align}\label{eq:l/l-1}
C(\sigma^{i_0+1}(\omega),\dots,\sigma^{i_3+1}(\omega))\,&=C(\sigma(\rho),\sigma^2(\rho),\sigma^3(\rho),\sigma^0(\rho))\\
&=C_{2341}(\rho,\sigma(\rho),\sigma^2(\rho),\sigma^3(\rho))\cr
&=\frac l{l-1}\kern4.8em \text{(by \eqref{C1234})}.\nonumber
\end{align}
On the other hand, $C(\sigma^{i_0+1}(\omega),\dots,\sigma^{i_3+1}(\omega))$ can be computed directly from \break
$C(\sigma^{i_0}(\omega),\dots,\sigma^{i_3}(\omega))$ ($=l$). 
Let $S_4$ acts on ${\overline\f_q}\!\!^4$ by permuting the components: $(x_1,\dots,x_4)^\pi=(x_{\pi(1)},\dots,x_{\pi(4)})$ for $\pi\in S_4$ and $(x_1,\dots,x_4)\in{\overline\f_q}\!\!^4$. (This is a right action.)
Write
\[
(\sigma^{i_0}(\omega),\dots,\sigma^{i_3}(\omega))=(\omega,\dots,\sigma^3(\omega))^\pi=(\sigma^{\pi(1)-1}(\omega),\dots,\sigma^{\pi(4)-1}(\omega)),
\]
where $\pi\in S_4$. Let $\tau=(1,2,3,4)\in S_4$. Then
\begin{align*}
(\sigma^{i_0+1}(\omega),\dots,\sigma^{i_3+1}(\omega))\,&=
(\sigma^{\pi(1)}(\omega),\dots,\sigma^{\pi(4)}(\omega))\cr
&=(\sigma^1(\omega),\dots,\sigma^4(\omega))^\pi\cr
&=(\omega,\dots,\sigma^3(\omega))^{\tau\pi}\cr
&=(\sigma^{i_0}(\omega),\dots,\sigma^{i_3}(\omega))^{\pi^{-1}\tau\pi}.
\end{align*}
Hence 
\[
C(\sigma^{i_0+1}(\omega),\dots,\sigma^{i_3+1}(\omega))=((\pi^{-1}\tau\pi)(C))(\sigma^{i_0}(\omega),\dots,\sigma^{i_3}(\omega)).
\]
In the above, $(\pi^{-1}\tau\pi)(C)=C_{j_1j_2,j_3,j_4}$, where $j_i=(\pi^{-1}\tau\pi)(i)$, $1\le i\le 4$.
Consider a Sylow 2-subgroup $G=\langle V,\,(1,2,3,4)\rangle=\langle V,\,(1,3)\rangle$ of $S_4$. Using \eqref{C1234}, it is routine, though rather tedious, to check that
\[
(\pi^{-1}\tau\pi)(C)=
\begin{cases}
\displaystyle\frac C{C-1}&\text{if}\ \pi\in G,\vspace{0.6em}\cr
\displaystyle\frac 1C\ \text{or}\ 1-C &\text{if}\ \pi\notin G.
\end{cases}
\]
Therefore,
\begin{equation}\label{eq:Ci0+1}
C(\sigma^{i_0+1}(\omega),\sigma^{i_1+1}(\omega),\sigma^{i_2+1}(\omega),\sigma^{i_3+1}(\omega))=
\begin{cases}
\displaystyle\frac l{l-1}&\text{if}\ \pi\in G,\vspace{0.6em}\cr
\displaystyle\frac 1 l\ \text{or}\ 1-l &\text{if}\ \pi\notin G.
\end{cases}
\end{equation}
By Lemma~\ref{L6.2} below, $l/(l-1)\notin\{1/l,\,1-l\}$. Hence by \eqref{eq:l/l-1} and \eqref{eq:Ci0+1}, $\pi\in G$, that is
\begin{align}\label{or}
(\sigma^{i_0}(\omega),\sigma^{i_1}(\omega),\sigma^{i_2}(\omega),\sigma^{i_3}(\omega))=\,&(\sigma^{i}(\omega),\sigma^{i+1}(\omega),\sigma^{i+2}(\omega),\sigma^{i+3}(\omega))\\
&\text{or}\ (\sigma^{i}(\omega),\sigma^{i-1}(\omega),\sigma^{i-2}(\omega),\sigma^{i-3}(\omega))\nonumber
\end{align}
for some $i$. For $\phi\in\overline\f_q(X)$ and $x_1,\dots,x_4\in\overline\f_q\cup\{\infty\}$, we write $\phi(x_1,\dots,x_4)=(\phi(x_1),\dots,\phi(x_4))$. Combining \eqref{beta} and \eqref{or} gives
\[
\beta(\rho,\dots,\sigma^3(\rho))=(\sigma^i(\omega),\dots,\sigma^{i+3}(\omega))\ \text{or}\ (\sigma^i(\omega),\dots,\sigma^{i-3}(\omega)).
\]

\medskip
{\bf Case 1.} Assume that $\beta(\rho,\dots,\sigma^3(\rho))=(\sigma^i(\omega),\dots,\sigma^{i+3}(\omega))$. 
For $0\le j\le 3$, we have
\[
(\sigma\circ\beta)(\sigma^j(\rho))=\sigma(\sigma^{j+i}(\omega))=\sigma^{j+i+1}(\omega)=\beta(\sigma^{j+1}(\rho))=(\beta\circ\sigma)(\sigma^j(\rho)).
\]
Hence $\sigma\circ\beta=\beta\circ\sigma$. On the other hand, we know that $\sigma\circ\beta=\bar\beta\circ\sigma$, where $\bar\beta$ is the result of $\sigma$ action on the coefficients of $\beta$. Therefore, $\bar\beta=\beta$, hence $\beta\in\pgl(2,\f_q)$. It follows easily that $\alpha\in\pgl(2,\f_q)$, and we are done.

\medskip
{\bf Case 2.} Assume that $\beta(\rho,\dots,\sigma^3(\rho))=(\sigma^i(\omega),\dots,\sigma^{i-3}(\omega))$. Then 
\[
C(\rho,\dots,\sigma^3(\rho))=C(\sigma^{i-3}(\omega),\dots,\sigma^{i}(\omega)).
\]
Hence there exists $\beta'\in\pgl(2,\overline \f_q)$ such that
\[
\beta'\bigl(\rho,\dots,\sigma^3(\rho)\bigr)=\bigl(\sigma^{i-3}(\omega),\dots,\sigma^{i}(\omega)\bigr).
\]
As in Case 1, $\beta'$ commute with $\sigma$, whence $\beta'\in\pgl(2,\f_q)$. 

Since $f=\alpha\circ g\circ\beta$, we have 
\begin{equation}\label{eq:f-sigma}
\bigl(f(\rho),\dots,f(\sigma^3(\rho))\bigr)=\alpha\bigl(g(\beta(\rho)),\dots, g(\beta(\sigma^3(\rho)))\bigr)=\alpha\bigl(g(\sigma^i(\omega)),\dots,g(\sigma^{i-3}(\omega))\bigr).
\end{equation}

\medskip
{\bf Case 2.1.} Assume that $f(\sigma^0(\rho)),\dots,f(\sigma^3(\rho))$ are distinct. From \eqref{eq:f-sigma},
\[
C\bigl(f(\rho),\dots,f(\sigma^3(\rho))\bigr)=C\bigl(g(\sigma^{i-3}(\omega)),\dots,g(\sigma^{i}(\omega))\bigr).
\]
Hence there exist $\alpha'\in\pgl(2,\overline \f_q)$ such that
\begin{equation}\label{eq:alpha'}
\bigl(f(\rho),\dots,f(\sigma^3(\rho))\bigr)=\alpha'\bigl(g(\sigma^{i-3}(\omega)),\dots,g(\sigma^{i}(\omega))\bigr).
\end{equation}
Both $f$ and $\alpha'\circ g\circ \beta'$ map $(\rho,\dots,\sigma^3(\rho))$ to $\bigl(f(\rho),\dots,f(\sigma^3(\rho))\bigr)$. By Lemma~\ref{T6.1} below, $f=\alpha'\circ g\circ\beta'$. It follows that $\alpha'\in\pgl(2,\f_q)$.

\medskip
{\bf Case 2.2.} Assume that $f(\rho),\dots,f(\sigma^3(\rho))$ are not all distinct. In this case, we only have to show that there also exits $\alpha'\in\pgl(2,\overline\f_q)$ such that $\eqref{eq:alpha'}$ holds. 

Note that $f(\sigma^j(\rho))=\sigma^j(f(\rho))$ (we define $\sigma(\infty)=\infty$). It follows that $\sigma^2(f(\rho))=f(\rho)$, hence
\[
\bigl(f(\rho),\dots,f(\sigma^3(\rho))\bigr)=
\bigl(f(\rho),f(\sigma(\rho)),f(\rho),f(\sigma(\rho))\bigr).
\]
By \eqref{eq:f-sigma}, we also have
\[
\bigl(g(\sigma^{i-3}(\omega)),\dots,g(\sigma^{i}(\omega))\bigr)=
\bigl(g(\sigma^{i-3}(\omega)),g(\sigma^{i-2}(\omega)),g(\sigma^{i-3}(\omega)),g(\sigma^{i-2}(\omega))\bigr),
\]
and $f(\rho)=f(\sigma(\rho))$ if and only if $g(\sigma^{i-3}(\omega))=g(\sigma^{i-2}(\omega))$. Choose $\alpha'\in\pgl(2,\overline\f_q)$ such that $\alpha'(f(\rho))=g(\sigma^{i-3}(\omega))$ and $\alpha'(f(\sigma(\rho)))=g(\sigma^{i-2}(\omega))$. Then \eqref{eq:alpha'} holds.
\end{proof}

\begin{lem}\label{L6.2}
Assume $p\ge 5$. Assume that $f_{s,t}$ ($s\ne 0$) has four ramification points $\rho,\sigma(\rho),\sigma^2(\rho),\sigma^3(\rho)$, where $\rho\in\f_{q^4}\setminus\f_{q^2}$. Let $(x_1,\dots,x_4)=(\rho,\dots,\sigma^3(\rho))$, and let
\[
l=C(x_1,x_2,x_3,x_4).
\]
Then $l\in\f_{q^2}$ and $l/(l-1)\notin\{1/l,\,1-l\}$.
\end{lem}

\begin{proof}
We have
\[
\sigma^2(l)=C(x_3,x_4,x_1,x_2)=C(x_1,x_2,x_3,x_4)=l,
\]
whence $l\in\f_{q^2}$.

Assume to the contrary that $l/(l-1)\in\{1/l,\,1-l\}$. It follows that
\begin{equation}\label{lambda-quadraic}
l^2-l+1=0.
\end{equation}
We claim that $l\notin\f_q$. Otherwise,
\[
l=C(x_1,x_2,x_3,x_4)=C(x_2,x_3,x_4,x_1)=\frac l{l-1},
\]
so $l=2$. Then \eqref{lambda-quadraic} becomes $3=0$, which is a contradiction. Now it follows from \eqref{lambda-quadraic} that $l\,l^q=1$. We have
\[
0=1-l\,l^q=\frac{h(x_1,x_2,x_3,x_4)}{(x_1-x_2)(x_2-x_3)(x_1-x_4)(x_3-x_4)},
\]
where
\[
h(x_1,x_2,x_3,x_4)=(x_1-x_3)^2(x_2-x_4)^2-(x_1-x_2)(x_2-x_3)(x_1-x_4)(x_3-x_4).
\]
$h(x_1,x_2,x_3,x_4)$ is symmetric in $x_1,\dots,x_4$, and hence it can be expressed as a polynomial in the elementary symmetric functions $s_1,\dots,s_4$ in $x_1,\dots,x_4$:
\[
h(x_1,\dots,x_4)=s_2^2-3s_1s_3+12s_4.
\]
Recall that $s_1=2$, $s_2=-s$, $s_3=2t$, $s_4=t$, which gives
\[
h(x_1,\dots,x_4)=s^2.
\]
Thus $s=0$, which is a contradiction.
\end{proof}

\begin{lem}\label{T6.1}
Let $F$ be any field, $f,g\in F(X)$, $\deg f=\deg g=r$. Assume that $a_1,\dots,a_k, b_1,\dots,b_k\in \overline F\cup\{\infty\}$ are such that $a_1,\dots,a_k$ are distinct, $a_i$ is a ramification point of both $f$ and $g$ with index $e_i$, and $f(a_i)=g(a_i)=b_i$. If $e_1+\cdots+e_k>2r$, then $f=g$.
\end{lem}

\begin{proof}
There exist $\alpha,\beta\in\pgl(2,\overline F)$ such that $\infty\notin\beta(\{a_1,\dots,a_k\})$ and $\infty\notin\alpha(\{b_1,\dots,b_k\})$. Replacing $f$ and $g$ by $\alpha\circ f\circ\beta^{-1}$ and $\alpha\circ g\circ\beta^{-1}$, $a_i$ by $\beta(a_i)$, and $b_i$ by $\alpha(b_i)$, we may assume $\infty\notin\{a_1,\dots,a_k,b_1,\dots,b_k\}$. Write $f(X)=P(X)/Q(X)$ and $g(X)=P_1(X)/Q_1(X)$, where $P,Q,P_1,Q_1\in\overline F(X)$, $\gcd(P,Q)=1=\gcd(P_1,Q_1)$. Then for all $1\le i\le k$,
\[
P(X)-b_iQ(X)\equiv 0\pmod{(X-a_i)^{e_i}},
\]
\[
P_1(X)-b_iQ_1(X)\equiv 0\pmod{(X-a_i)^{e_i}}.
\]
It follows that
\[
P(X)Q_1(X)\equiv P_1(X)Q(X)\pmod{(X-a_i)^{e_i}}.
\]
Thus
\[
P(X)Q_1(X)\equiv P_1(X)Q(X)\pmod{\prod_{i=1}^k(X-a_i)^{e_i}}.
\]
Since $e_1+\cdots+e_k>2r$, we have
\[
P(X)Q_1(X)=P_1(X)Q(X),
\]
that is, $f=g$.
\end{proof}

\begin{thm}\label{T6.3}
Let $q$ be arbitrary. Assume that $f,g$ are Case III-b such that $\theta(f)=\theta(g)$ and $\lambda(f)=\lambda(g)\notin\{1,5/2\}$. Then $f\sim g$.
\end{thm}

\begin{proof}
Consider a Sylow 2-subgroup $H=\langle V,\, (1,2)\rangle$ of $S_4$. For each $\pi\in S_4$, we have
\begin{equation}\label{pi}
\pi\Bigl(C+\frac 1C\Bigr)=
\begin{cases}
\displaystyle C+\frac 1C&\text{if}\ \pi\in H,\vspace{0.4em}\cr
\displaystyle \frac C{C-1}+\frac {C-1}C&\text{if}\ \pi\in (1,3)H,\vspace{0.4em}\cr
\displaystyle 1-C+\frac 1{1-C}&\text{if}\ \pi\in (1,4)H.
\end{cases}
\end{equation}

Since $\theta(f)=\theta(g)$, we have $f\sim g$ over $\overline\f_q$, hence there exist $\alpha,\beta\in\pgl(2,\overline\f_q)$ such that $f=\alpha\circ g\circ\beta$. Let $\rho_1,\rho_1^q,\rho_2,\rho_2^q$ be the ramification points of $f$ and $\omega_1,\omega_1^q,\omega_2,\omega_2^q$ be the ramification points of $g$, and let $l=C(\rho_1,\rho_1^q,\rho_2,\rho_2^q)$ and $l'=C(\omega_1,\omega_1^q,\omega_2,\omega_2^q)$. (Clearly, $l,l'\in\f_q$.)
Note that $\beta(\rho_1,\rho_1^q,\rho_2,\rho_2^q)$ is a permutation of $(\omega_1,\omega_1^q,\omega_2,\omega_2^q)$. Write $(\omega_1,\omega_1^q,\omega_2,\omega_2^q)=(x_1,x_2,x_3,x_4)$ and $\beta(\rho_1,\rho_1^q,\rho_2,\rho_2^q)=(x_{\pi(1)},\dots,x_{\pi(4)})$, where $\pi\in S_4$. Then
\[
l=C(\rho_1,\rho_1^q,\rho_2,\rho_2^q)=C(x_{\pi(1)},\dots,x_{\pi(4)}).
\]
Since $\lambda(g)=\lambda(f)$, using \eqref{pi}, we have 
\begin{align}\label{u+1/u}
{l'}+\frac 1{l'}\,&=l+\frac 1 l=C(x_{\pi(1)},\dots,x_{\pi(4)})+C(x_{\pi(1)},\dots,x_{\pi(4)})^{-1}\\
&=\begin{cases}
\displaystyle {l'}+\frac 1{l'}&\text{if}\ \pi\in H=\langle V,(1,2)\rangle,\vspace{0.4em}\cr
\displaystyle \frac {l'}{{l'}-1}+\frac {{l'}-1}{l'}&\text{if}\ \pi\in (1,3)H,\vspace{0.4em}\cr
\displaystyle 1-{l'}+\frac 1{1-{l'}}&\text{if}\ \pi\in (1,4)H,
\end{cases}\nonumber
\end{align}
where the last equality follows from \eqref{pi}. Note that
\begin{align*}
&{l'}+\frac 1{l'}\in\Bigl\{\frac{l'}{{l'}-1}+\frac {{l'}-1}{l'},\, 1-{l'}+\frac 1{1-{l'}}\Bigr\}\cr
\Leftrightarrow\ &{l'}\in\Bigl\{\frac{l'}{{l'}-1},\,\frac {{l'}-1}{l'},\, 1-{l'},\,\frac 1{1-{l'}}\Bigr\}\cr
\Leftrightarrow\ &{l'}\in\{2,\,1/2,\,-\epsilon,\,-\epsilon^{-1}\}\quad\text{where $o(\epsilon)=3$}\cr
\Leftrightarrow\ & {l'}+\frac 1{l'}\in\{1,\,5/2\}.
\end{align*}
Therefore, by assumption,
\[
{l'}+\frac 1{l'}\notin\Bigl\{\frac{l'}{{l'}-1}+\frac {{l'}-1}{l'},\, 1-{l'}+\frac 1{1-{l'}}\Bigr\}.
\]
It then follows from \eqref{u+1/u} that $\pi\in H$. A permutation of $(\omega_1,\omega_1^q,\omega_2,\omega_2^q)$ by $\pi$ is of the form $(\nu_1,\nu_1^q,\nu_2,\nu_2^q)$. We have
\[
\beta(\rho_1,\rho_1^q,\rho_2,\rho_2^q)=(\nu_1,\nu_1^q,\nu_2,\nu_2^q),
\]
and it follows that 
$\beta$ commute with $\sigma=(\ )^q$. Hence $\beta\in\pgl(2,\f_q)$. Consequently, it follows from $f=\alpha\circ g\circ\beta$ that $\alpha\in\pgl(2,\f_q)$.
\end{proof}

\section{Examples}

We consider two examples with $q=3^3$ and $5^2$, respectively. In each case, we enumerate the representatives of the equivalence classes of degree 3 rational functions explicitly.

\begin{exmp}\rm
Let $\f_{3^3}=\f_3(\alpha)$, where $\alpha^3-\alpha+1=0$. A set of representatives of the equivalence classes of degree 3 rational functions over $\f_{3^3}$ is given below.

\medskip
Class I: $X^3$, $X^3-\alpha X$.

\medskip
Class II: $\displaystyle\frac{X^3+1}X$;  $\displaystyle\frac{X^3+X^2+t}X,\ t\in\f_{3^3}^*$.

\medskip
Class III-a: $f_{s,t}$, where $(s,t)$ and the corresponding values of $\theta$ are given in Table~\ref{tb:3^3-case3a}.

\begin{table}[ht]
\caption{$q=3^3$, Class III-a: values of $(s,t)$ and $\theta$}\label{tb:3^3-case3a}
   \renewcommand*{\arraystretch}{1.4}
    \centering
     \begin{tabular}{c|c}
         \hline
         $(s,t)$ & $\theta$ \\ \hline
        $(\A^2, \A^2)$ & $2 + \A + 2 \A^2$  \\
        $(1 + \A + 2 \A^2, \A^2)$ & $2 + \A$  \\
        $(1 + 2 \A, \A^2)$ & $2$ \\
        $(2, \A^2)$ & $1 + \A$ \\
        $(2 + \A, \A^2)$ & $2 \A + 2 \A^2$ \\
        $(2 + 2 \A + 2 \A^2, \A^2)$ & $ 2 \A^2$ \\
        $(2 \A + \A^2, 2 \A^2)$ & $2 + 2 \A + \A^2$ \\
        $(1 + 2 \A + \A^2, 2 \A^2)$ & $\A + 2 \A^2$ \\
        $(\A^2, \A)$ & $1 + 2 \A^2$ \\
        $(\A + \A^2, \A)$ & $2 + \A + \A^2$ \\
        $(2 + \A, \A + \A^2)$ & $1 + \A^2$ \\
        $(2 \A, 2 \A)$ & $2 + 2 \A + 2 \A^2$ \\ 
        $(1 + 2 \A + 2 \A^2, 2 \A)$ & $\A$ \\ 
         \hline
     \end{tabular}
\end{table}

\medskip
Class III-b: $f_{s,t}$, where $(s,t)$ and the corresponding values of $\mu^2$ are given in Table~\ref{tb:3^3-case3b}.

\begin{table}[ht]
\caption{$q=3^3$, Class III-b: values of $(s,t)$ and $\mu^2$}\label{tb:3^3-case3b}
   \renewcommand*{\arraystretch}{1.4}
    \centering
     \begin{tabular}{c|c}
         \hline
         $(s,t)$ & $\mu^2$ \\ \hline
         $(\A+\A^2, \A^2)$ & $2+2\A$ \\
         $(1+2\A^2,\A^2)$ & $\A^2$ \\
         $(2+2\A,\A^2)$& $\A+\A^2$ \\
         $(\A,2\A^2)$& $1+2\A$ \\
         $(1+\A+\A^2,2\A^2)$& $1+\A+2\A^2$ \\
         $(1+2\A+2\A^2,2\A^2)$& $2\A$ \\
         $(2\A^2,\A)$& $2\A+\A^2$ \\
         $(2\A,\A+2\A^2)$& $1+2\A+2\A^2$ \\
         $(2+\A+2\A^2,\A+2\A^2)$& $1+2\A+\A^2$ \\
         $(2+2\A,2\A+\A^2)$& $2+2\A^2$ \\
         $(1+2\A,2\A+2\A^2)$& $2+\A^2$ \\
         $(2+2\A+2\A^2,2\A+2\A^2)$& $1+\A+\A^2$ \\
         \hline
     \end{tabular}
\end{table}

\medskip
Class III-d: $f_{0,1}$.
\end{exmp}

\begin{exmp}\rm
Let $\f_{5^2}=\f_5(\alpha)$, where $\alpha^2+\alpha+1=0$. A set of representatives of the equivalence classes of degree 3 rational functions over $\f_{5^2}$ is given below.

\medskip
Class I: $\displaystyle\frac{X^3+(1-2\alpha)X}{(1-\alpha)X^2+1}$.

\medskip
Class II: $\displaystyle\frac{X^3+\alpha^i}X,\ i=0,1,2$; $\displaystyle{X^3+X^2+t}X,\ t\in\f_{5^2}^*$.

\medskip
Class III-a: $f_{s,t}$, where $(s,t)$ and the corresponding values of $\theta$ are given in Table~\ref{tb:5^2-case3a}.

\begin{table}[ht]
\caption{$q=5^2$, Class III-a: values of $(s,t)$ and $\theta$}\label{tb:5^2-case3a}
   \renewcommand*{\arraystretch}{1.4}
    \centering
     \begin{tabular}{c|c}
         \hline
         $(s,t)$ & $\theta$ \\ \hline
         $(\A, \A)$ & $3+2\A$ \\
         $(4\A,\A)$ & $1+\A$ \\
         $(1,\A)$& $1+3 \A$ \\
         $(2+3\A,\A)$& $\A$ \\
         $(3+2\A,\A)$& $4+4\A$ \\
         $(4,\A)$& $4\A$ \\
         $(1+3\A,2\A)$& $4+\A$ \\
         $(2,2\A)$& $4+3 \A$ \\
         $(\A,3\A)$& $1+2\A$ \\
         $(4+3\A,3\A)$& $3+4\A$ \\
         $(3\A,1+3\A)$& $2\A$ \\
         $(1+2\A,1+3\A)$& $3+3 \A$ \\ \hline
     \end{tabular}
\end{table}

\medskip
Class III-b: $f_{s,t}$, where $(s,t)$ and the corresponding values of $\mu^2$ are given in Table~\ref{tb:5^2-case3b}.

\begin{table}[ht]
\caption{$q=5^2$, Class III-b: values of $(s,t)$ and $\mu^2$}\label{tb:5^2-case3b}
   \renewcommand*{\arraystretch}{1.4}
    \centering
     \begin{tabular}{c|c}
         \hline
         $(s,t)$ & $\mu^2$ \\ \hline
         $(2+4\A, \A)$ & $\A$ \\
         $(4+2\A,\A)$ & $4+4\A$ \\
         $(2+2\A,2\A)$& $1+ \A$ \\
         $(3+4\A,2\A)$& $3\A$ \\
         $(4+\A,2\A)$& $2\A$ \\
         $(1+\A,3\A)$& $4\A$ \\
         $(2+4\A,3\A)$& $2+2\A$ \\
         $(3+2\A,3\A)$& $3+3 \A$ \\
         $(0,4\A)$& $2$ \\
         $(1+4\A,4\A)$& $3$ \\\hline
     \end{tabular}
\end{table}

\medskip
Classes III-b and III-d: $f_{1+2\alpha,2\alpha}$, $f_{-3,1}$.

\end{exmp}

\appendix 
\section{Even Characteristic}\label{app-a}

Our approach also works in characteristic 2. Classes I and II do not depend on the characteristic. We will see that equivalence classes in Class III are determined by the invariant $\theta$. We assume that $q$ is even throughout this section. 

\begin{prop}\label{prop:p=2}\
\begin{itemize}
\item[(i)]
$f_{0,t}$ belongs to Case II.

\medskip
\item[(ii)] For $s\ne 0$, $f_{s,t}$ belongs to Case II if and only if $\text{\rm Tr}_{q/2}(t/s^2)=0$ and $(s,t)\ne (1,1)$, where $\text{\rm Tr}_{q/2}$ is the trace from $\f_q$ to $\f_2$. 
\end{itemize}
\end{prop}

\begin{proof}
(i)
In this case, 
\[
X^4-2X^3-sX^2-2tX+t=X^4+t=(X+a)^4,
\]
where $a\in\f_q$ is such that $a^4=t$. By Theorem~\ref{Theorem 2}, it suffices to show that $a^3\ne t$. If $a^3=t$, then $t=1$, whence $1+s+t=0$, which is a contradiction.

\medskip
(ii).
It is well known that $X^4+sX^2+t$ has a root in $\f_q$ if and only if $\text{Tr}_{q/2}(t/s^2)=0$. Moreover, $X^4+sX^2+t$ has a root $x$ with $x^3\ne t$ if and only if $(X^3+t)^2\not\equiv0\pmod{X^4+sX^2+t}$. We have
\[
(X^3+t)^2\equiv(s^2+t)X^2+t(s+t)\pmod{X^4+sX^2+t}.
\]
It follows that $(X^3+t)^2\not\equiv0\pmod{X^4+sX^2+t}$ if and only if $(s,t)\ne (1,1)$. Now statement (ii) follows from Theorem~\ref{Theorem 2}.
\end{proof}

\begin{thm}\label{thm:p=2}
For each $b\in\f_q\setminus\f_2$, there exist $s,t\in\f_q$ with $t(1+s+t)\ne0$ and $\text{Tr}_{q/2}(t/s^2)=1$ such that
$\theta(s,t)=b$.
\end{thm}

The proof of Theorem~\ref{thm:p=2} uses the following result on exponential sums over elliptic curves:

\begin{thm}\cite[Corollary~1]{Kohel-Shparlinski-2000}\label{thm:ellip}
Let $E$ be an elliptic curve over $\f_q$, $\f_q(E)$ be its function field and $E(\f_q)$ be the set of rational points on $E$. Let $\chi$ be a nontrivial additive character of $\f_q$ and $f\in\f_q(E)$ with $\deg f>0$. Then 
\[
\Bigl|\sum_{\substack{P\in E(\f_q)\cr f(P)\ne\infty}}\chi(f(P))\Bigr|\le 2(\deg f)q^{1/2}.
\]
\end{thm}

\begin{proof}[Proof of Theorem~\ref{thm:p=2}]
Let 
\[
g(X,Y)=X^3-bY(1+X+Y)\in\f_q[X,Y].
\]
We have
\[
\frac{\partial g}{\partial X}=X^2-bY,\quad
\frac{\partial g}{\partial Y}=b(1+X).
\] 
Since $b\ne 0,1$, it is easy to see that the system
\[
\begin{cases}
g(x,y)=0,\vspace{0.4em}\cr
\displaystyle\frac{\partial g}{\partial X}(x,y)=0,\vspace{0.4em}\cr
\displaystyle\frac{\partial g}{\partial Y}(x,y)=0
\end{cases}
\]
has no solution. Therefore, the curve defined by the degree 3 polynomial $g$ is nonsingular and hence is an elliptic curve over $\f_q$, which is denoted by $E$.

Let $N_0, N_1$ denote the number of points $(s,t)\in E(\f_q)$ with $\trace_{q/2}(t/s^2)=0,1$, respectively. Then $N_0+N_1+1=|E(\f_q)|$, where $1$ counts for the point of $E$ at infinity. By the Hasse-Weil bound, $|E(\f_q)|\ge q+1-2q^{1/2}$. Let $f=Y/X^2\in\f_q(E)$ and let $\chi(z)=(-1)^{\trace_{q/2}(z)}$ be the canonical additive character of $\f_q$. Set
\[
S=\sum_{\substack{P\in E(\f_q)\cr f(P)\ne\infty}}\chi(f(P))=N_0-N_1.
\]
By Theorem~\ref{thm:ellip}, $|S|\le 4q^{1/2}$. Hence we have
\[
N_1=\frac 12(|E(\f_q)|-1-S)\ge\frac 12(q-2q^{1/2}-4q^{1/2})=\frac 12q-3q^{1/2},
\]
which is positive when $q\ge 2^6$. For $q\le 2^5$, we also have $N_1>0$ via computer verification.
\end{proof}

\begin{rmk}\rm
There is a slightly different proof of Theorem~\ref{thm:p=2} without using Theorem~\ref{thm:ellip}. Choose $\epsilon\in\f_q$ such that $\trace_{q/2}(\epsilon)=1$, and show that the equation $s^3/t(1+s+t)=b$ has a solution $(s,t)\in\f_q^2$ with $t=s^2(a^2+a+\epsilon)$ for some $a\in\f_q$. The equation $\theta(s,t)=b$ is equivalent to 
$F(a,s(a^2+a+\epsilon))=0$,
where
\[
F(X,Y)=(X^2+X+\epsilon)Y^2+(X^2+X+\epsilon+b^{-1})Y+(X^2+X+\epsilon)^2\in\f_q[X,Y].
\]
One can show that $F(x,Y)$ is absolutely irreducible, and then use the Hasse-Weil bound to conclude that $F$ has a zero $(x,y)\in\f_q^2$.
\end{rmk}

\begin{rmk}\rm
The proofs of Theorem~\ref{thm:p=2} given here, whether using Theorem~\ref{thm:ellip} or not, are non-constructive. We do not know if one can explicitly construct some solution $(s,t)$ in Theorem~\ref{thm:p=2}.
\end{rmk}

\begin{thm}\label{thm:theta-onto}
For each $b\in\f_q^*$, there exists $(s,t)\in\f_q^2$ such that $f_{s,t}$ is in Class III and $\theta(s,t)=b$.
\end{thm}

\begin{proof}
If $b\ne 1$, the conclusion follows from Theorem~\ref{thm:p=2} and Proposition~\ref{prop:p=2} (ii). If $b=1$, we have $\theta(1,1)=1$, where $f_{1,1}$ is in Class III by Proposition~\ref{prop:p=2} (ii).
\end{proof}

\begin{cor}\label{cor:p=2}
The equivalence classes in Class III are determined by the invariant $\theta$.
\end{cor}

\begin{proof}
By Theorem~\ref{thm:theta-onto}, the invariant $\theta$ of functions in Class III takes all values in $\f_q^*$. (note that $\theta\ne 0$ by Proposition~\ref{prop:p=2} (i).) On the other hand, we know that there are precisely $q-1$ equivalence classes in Class III. Hence the conclusion.
\end{proof}

Table\ref{tb4} provides a comparison between the classifications in \cite{Mattarei-Pizzato-JAA-2024} and the current paper.
In Table~\ref{tb4}, $\sigma\in\f_q$ is a chosen element such that $\trace_{q/2}(\sigma)=1$,  $\mathcal C$ is a transversal of $(\f_q^*)^{(3)}$ in $\f_q^*$, and $g=(X^3+a_2X^2+a_1X)/(X^2+X+b_0)$, where $a_1,a_2,b_1\in\f_q$ are chosen elements such that $\trace_{q/2}(b_0)=1$, $a_1=b_0+b_0^{-1}$, $a_2=1+b_0^{-1}$.

\begin{table}[ht]
\caption{
A comparison between the classifications in \cite{Mattarei-Pizzato-JAA-2024} and the current paper for even $q$}\label{tb4}
   \renewcommand*{\arraystretch}{1.5}
    \vskip-1em
    \[\kern-5.5em
    {\small
     \begin{tabular}{c|l|c|c||c|l|l} \hline
     $\begin{array}{cc}\text{\cite{Mattarei-Pizzato-JAA-2024}}\vspace{-0.6em}\cr \text{cls} \end{array}$ &\hfil representative\hfil & ram type & $\theta$ & $\begin{array}{cc}\text{our}\vspace{-0.6em}\cr \text{cls}\end{array}$ &\hfil representative\hfil & condition \\ \hline
     (i) & $X^3$ & $(3/1,3/1)$ & 1 & I & $X^3$ & $q$ nonsquare \\
     & & & & III & $f_{1,1}$ & $q$ square \\ \hline
     (ii) & $\displaystyle\frac{X^3+\sigma X+\sigma}{X^2+X+\sigma+1}$ & $(3/2,3/2)$ &1& III & $f_{1,1}$ & $q$ nonsquare \\
     &  & & & I & $g$ & $q$ square\\ \hline
     (iii) & $X^3+X^2$ & $(3/1,3/1)$ & 1 & II & $\displaystyle\frac{X^3+X^2+1}X$ & \\ \hline
     (iv) & $\displaystyle\frac{X^3+t}X,\ t\in\mathcal C$ & $(2/1)$ &0 & II &  $\displaystyle\frac{X^3+t}X,\ t\in\mathcal C$\\ \hline
     (v) &  $\displaystyle\frac{X^3+c}{X+c},\ c\in\f_q\setminus\f_2$ & $(2/1,2/1)$ & $(1+c)^2$ & II & $\displaystyle\frac{X^3+X^2+t}X$ \\
     &  &&&& $t=(1+c)^{-2}\in\f_q\setminus\f_2$ \\ \hline
     (vi) & $\displaystyle\frac{X^3+bX^2+\sigma X+(b+1)\sigma}{X^2+X+b+1+\sigma}$ & $(2/2,2/2)$ & $(b+1)^{-4}$ &III & $f_{s,t}$\\
     & $b\in\f_q\setminus\f_2$ &&&& $\trace_{q/2}(t/s^2)=1$\\
     &&&&& $\theta(f_{s,t})=(b+1)^{-4}\in\f_q\setminus\f_2$ \\ \hline     
      \end{tabular}
     }
     \]
\end{table}

\section{Mathematica Codes}\label{app-b}

\subsection{Proof of Theorem~\ref{C-tha1}: verification of \eqref{pc}}

\begin{mmaCell}{Code}

Clear[alpha, c, K, Pc, B2, B3, B4, B5, B6];
alpha = 8 - 15 c - a c + 6 c^2 + c^3;
K = 8 x^6 - 12 x^7 + 6 x^8 - x^9 - 24 x^4 y + 42 x^6 y - 30 x^7 y + 
   6 x^8 y - 30 x^2 y^2 - 2 a x^2 y^2 + 171 x^3 y^2 + 5 a x^3 y^2 - 
   180 x^4 y^2 - 4 a x^4 y^2 + 15 x^5 y^2 + a x^5 y^2 + 42 x^6 y^2 - 
   12 x^7 y^2 - 8 y^3 + 84 x y^3 + 4 a x y^3 - 330 x^2 y^3 - 
   14 a x^2 y^3 + 442 x^3 y^3 + 14 a x^3 y^3 - 180 x^4 y^3 - 
   4 a x^4 y^3 + 8 x^6 y^3 - 24 y^4 + 168 x y^4 + 8 a x y^4 - 
   330 x^2 y^4 - 14 a x^2 y^4 + 171 x^3 y^4 + 5 a x^3 y^4 - 
   24 x^4 y^4 - 24 y^5 + 84 x y^5 + 4 a x y^5 - 30 x^2 y^5 - 
   2 a x^2 y^5 - 8 y^6;
Pc = c y + 2 x^2 - (4 + c) x y + c y^2 - x^2 (x - 2 y);
B2 = -(8/c) y^2;
B3 = -(2/c) y (c (c + 6) x^2 - 2 (4 + 6 c + c^2 ) x y + 8 y^2);
B4 = (1/c) (4 c x^4 + c (2 + 3 c) x^3 y - 8 (1 + 4 c + c^2) x^2 y^2 + 
     4 (4 + 6 c + c^2 ) x y^3 - 8 y^4);
B5 = -x^2 (x - 2 y) (4 x^2 + (c - 2) x y - (c + 6) y^2);
B6 = x^4 (x - 2 y)^2;
K - Pc (B2 + B3 + B4 + B5 + B6);
PolynomialMod[c*

Out[1]= 0

\end{mmaCell}

\subsection{Proof of Theorem~\ref{thm:mu2}: computation and factorization of $N$}

\begin{mmaCell}{Code}

Clear[M, a1, a0, a11, a00, a, b, c, d, K];
M[a0_, a1_] := 
  Simplify[((a1^2 - 4 a0) (2 + 2 a0 + a1))/(a1 (2 + a1) (2 a0 + 
        a1))];
a00 = (b^2 + d^2 a0 + b d a1)/(a^2 + c^2 a0 + a c a1);
a11 = (2 a b + 2 c d a0 + (b c + a d) a1)/(a^2 + c^2 a0 + a c a1);
Factor[M[a0, a1] - M[a00, a11]];
NN = Numerator[
C0 = 2 b + 2 b^2 - 4 b c - 2 a1 b c - 2 b^2 c - a1 b^2 c + 2 b c^2 + 
   a1 b c^2 - 2 d - a1 d - 4 b d - 2 a1 b d + 4 c d + 2 a1 c d + 
   4 b c d + 2 a1 b c d + 4 a1 c^2 d + 4 a1^2 c^2 d + a1^3 c^2 d + 
   2 d^2 + a1 d^2 + 4 a1 c d^2 + 4 a1^2 c d^2 + a1^3 c d^2;
C1 = 2 b + 2 b^2 - 4 b c - 2 a1 b c - 2 b^2 c - a1 b^2 c + 2 b c^2 + 
   a1 b c^2 - 2 d - a1 d - 4 b d - 2 a1 b d + 4 c d + 2 a1 c d + 
   4 b c d + 2 a1 b c d - 2 c^2 d + a1 c^2 d + 3 a1^2 c^2 d + 
   a1^3 c^2 d + 2 d^2 + a1 d^2 - 2 c d^2 + a1 c d^2 + 3 a1^2 c d^2 + 
   a1^3 c d^2;
Factor[NN /. {a -> 1, a0 -> -a1 C0/(2 C1)}]

\end{mmaCell}


\section*{Acknowledgments}

This work was partially supported by NSF REU Grant 2244488 and NSF RTG grant 2342254.




\end{document}